\documentclass{article}

\usepackage{amsmath,amssymb,amscd,amsthm,diagrams,fullpage}
\input cyracc.def

\newcommand{\QQ}{\mathbb Q}
\renewcommand{\AA}{\mathbb A}
\newcommand{\ZZ}{\mathbb Z}
\newcommand{\NN}{\mathbb N}

\newcommand{\RR}{\mathbb R}

\newcommand{\BB}{\mathbb B}
\newcommand{\CC}{\mathbb C}
\newcommand{\DD}{\mathbb D}
\newcommand{\EE}{\mathbb E}

\newcommand{\calC}{{\cal C}}

\newcommand{\calT}{{\cal T}}
\newcommand{\calV}{{\cal V}}
\newcommand{\calB}{{\cal B}}
\newcommand{\calZ}{{\cal Z}}

\newcommand{\mun}{\mu_{p^n}}

\newcommand{\vs}{\vspace{1ex}}

\newcommand{\calA}{{\cal A}}

\newcommand{\calK}{{\cal K}}
\newcommand{\calH}{{\cal H}}
\newcommand{\calG}{{\cal G}}

\newcommand{\Omegahat}{\widehat{\Omega}}

\DeclareMathOperator{\coker}{coker}
\DeclareMathOperator{\Mod}{mod}

\DeclareMathOperator{\Gal}{Gal}

\DeclareMathOperator{\image}{Im}
\DeclareMathOperator{\id}{id}

\DeclareMathOperator{\Tr}{Tr}

\DeclareMathOperator{\Hom}{Hom}

\DeclareMathOperator{\MOD}{Mod}
\DeclareMathOperator{\Rep}{Rep}
\DeclareMathOperator{\dR}{dR}

\DeclareMathOperator{\Fil}{Fil}

\DeclareMathOperator{\GL}{GL}
\DeclareMathOperator{\cd}{cd}

\DeclareMathOperator{\et}{\acute{e}t}

\DeclareMathOperator{\fil}{fil}

\newtheorem{thm}{Theorem}[section]
\newtheorem{prop}[thm]{Proposition}
\newtheorem{lem}[thm]{Lemma}
\newtheorem{cor}[thm]{Corollary}
\newtheorem{remark}[thm]{Remark}

\begin{document}

\title{Bloch-Kato exponential maps for local fields with imperfect residue fields}

\author{Sarah Livia Zerbes}
\maketitle

\begin{center}
 {\it email:} s.zerbes@imperial.ac.uk
\end{center}

\begin{abstract}
 In this paper, we generalise the construction of the Bloch-Kato exponential map to complete discrete valuation fields of mixed characteristic $(0,p)$ whose residue fields have a finite $p$-basis. As an application we prove an explicit reciprocity law, generalising Th\'eor\`eme IV.2.1 in~\cite{cherbonniercolmez99}. This result relies on the calculation of the Galois cohomology of a $p$-adic representation $V$ in terms of its $(\phi,G)$-module. 
\end{abstract}

\tableofcontents 

%++++++++++++++++++++++++++++++++

 \section{Introduction}

  \subsection{Statement of the main results}
  
   Throughout this paper, let $p$ be an odd prime. Let $\BB_{\dR}$ and $\BB_{\max}$ be the period rings of Fontaine~\cite{fontaine90}, and let $V$ be a de Rham representation of $\calG_F$, where $F$ is a finite extension of $\QQ_p$ and $\calG_F$ denotes its absolute Galois group. In~\cite{blochkato90}, Bloch and Kato have constructed an exponential map
   \begin{equation}\label{classicalBK}
    \exp_{F,V}:\DD_{\dR}(V)\rightarrow H^1(F,V),
   \end{equation}
   which is obtained by taking $\calG_F$-cohomology of the fundamental exact sequence 
   \begin{equation*}
    0\rTo V\rTo V\otimes_{\QQ_p}\BB_{\max}^{\phi=1}\rTo V\otimes_{\QQ_p}\BB_{\dR}\slash \BB_{\dR}^+\rTo 0.
   \end{equation*}
   The Bloch-Kato exponential map is the composition of the connecting homomorphism $\DD_{\dR}(V)\slash\DD_{\dR}^+(V)\rightarrow H^1(F,V)$ and the quotient map $\DD_{\dR}(V)\rightarrow \DD_{\dR}(V)\slash\DD_{\dR}^+(V)$. If $V=\QQ_p\otimes_{\ZZ_p}T$, where $T$ is the $p$-adic Tate module of a $p$-divisible group $G$ defined over $F$, then $V$ is a de Rham representation of $\calG_F$, and $\exp_{F,V}$ agrees with the exponential map on the tangent space of $G$: as shown in Section 3.10.1 in~\cite{blochkato90}, there is a commutative diagram
   \begin{diagram}
    \tan(G(F)) & \rTo^{\exp_G} & \QQ\otimes_{\ZZ} G(O_F) \\
    \dTo^{\cong} & & \dTo^{\delta} \\
    \DD_{\dR}(V)\slash \DD_{\dR}^+(V) & \rTo^{\exp_{F,V}} & H^1(F,V)
   \end{diagram}
   where $\delta$ is the Kummer map. The cup product gives a perfect pairing
   \begin{equation*}
    H^1(F,V)\times H^1(F,V^*(1))\rTo^\cup H^2(F,\QQ_p(1))\cong \QQ_p,
   \end{equation*}
   and one can use it to `dualise'~\eqref{classicalBK} to obtain the dual exponential map
   \begin{equation*}
    \exp^*_{F,V}:H^1(F,V^*(1))\rTo \DD_{\dR}(V^*(1)).
   \end{equation*}
   An important application of the dual exponential map is the explicit reciprocity law of Cherbonnier and Colmez (Th\'eor\`eme IV.2.1 in~\cite{cherbonniercolmez99}): An unpublished theorem of Fontaine (c.f. Th\'eor\`eme II.1.3 in~\cite{cherbonniercolmez99}) states that the group $\varprojlim H^1(F_n,V)$, where the inverse limit is taken with respect to the corestriction maps in the cyclotomic tower, is canonically isomorphic to a submodule of the overconvergent $(\phi,\Gamma)$-module of $V$. Th\'eor\`eme IV.2.1 in~\cite{cherbonniercolmez99} shows that if $V$ is a de Rham representation of $\calG_F$, then the inverse of this isomorphism can be described in terms of the dual exponential map.
   \vs
   
   In this paper, we generalise the Bloch-Kato exponential map to discrete valuation fields of mixed characteristic with imperfect residue field. Let $L$ be such a field with residue field $k_L$. Assume that $[k_L:k_L^p]=p^{d-1}$. Kato~\cite{kato99} and Brinon~\cite{brinon06} have defined  analogues of Fontaine's rings of periods for the imperfect residue field case, and in~\cite{brinon06} Brinon shows the existence of a  fundamental exact sequence 
   \begin{equation}\label{fd}
    0  \rTo V\rTo V\otimes_{\QQ_p}\BB_{\max}^{\nabla \phi=1}\rTo V\otimes_{\QQ_p}\BB_{\dR}^\nabla\slash \BB_{\dR}^{\nabla +}\rTo 0,
   \end{equation}
   where $\BB_{\dR}^\nabla\slash \BB_{\dR}^{\nabla+}$ has the exact resolution
   \begin{align*}
    0 & \rTo V\otimes_{\QQ_p}\BB_{\dR}^\nabla\slash\BB_{\dR}^{\nabla +} \rTo V\otimes_{\QQ_p}\BB_{\dR}\slash\Fil^0\BB_{\dR} \rTo^\nabla V\otimes_{\QQ_p}\BB_{\dR}\slash\Fil^{-1}\BB_{\dR}\otimes_K\Omega^1_K \\
    & \rTo^\nabla\dots\rTo^\nabla V\otimes_{\QQ_p}\BB_{\dR}\slash\Fil^{1-d}\BB_{\dR}\otimes_K\Omega^{d-1}_K \rTo 0.    
   \end{align*}

   In Section~\ref{exponentials+duals}, we define de Rham cohomology groups $H^i_{\dR}(V)$ for any $p$-adic representation $V$ of $\calG_L$, and we show that~\eqref{fd} gives rise to exponential maps
   \begin{equation}\label{higherexp}
    \exp_{(i),L,V}: H^{i-1}_{\dR}(V)\oplus H^i_{\dR}(V) \rTo H^{i+1}(L,V)
   \end{equation}
   for $1\leq i\leq d$. Analogous to the classical case, one can dualise~\eqref{higherexp} to obtain dual exponential maps
   \begin{equation*}
    \exp^*_{(i),K,V}: H^{d+1-i}(L,V^*(d)) \rTo H_{\dR}^{d-i}(V^*(d))\oplus H_{\dR}^{d-i-1}(V^*(d)).
   \end{equation*}   
   In the case when $V\cong \QQ_p\otimes_{\ZZ_p}T$, where $T$ is the $p$-adic Tate module of a $p$-divisible group $\frak{G}$ defined over $L$, the Hodge-Tate weights of $V$ are contained in the set $\{0,1\}$, which implies that $\exp_{(i),L,V}$ and $\exp^*_{(i),L,V}$ are the zero maps unless $i=0,1$. Furthermore, one can show that $\exp_{(0),L,V}$ agrees via the Kummer map with the exponential map on the tangent space of $\frak{G}$. We will study this case in detail in a forthcoming paper. 
   \vs
   
   As an application of this construction, we prove the following result, which can be seen as a generalisation of Th\'eor\`eme II.1.3 in~\cite{cherbonniercolmez99}.
   
   \begin{thm}\label{samecohomology*}
    Let $V$ be a de Rham representation of $\calG_L$. Then for $n\gg 0$ and for all $1\leq i\leq d$, we have a commutative diagram
    \begin{diagram}
     H^i(G_L,\DD^{\dagger,n}(V)^{\psi=1}) & \rTo^{\phi^{-n}} & H^i(L,\BB_{\dR}^\nabla\otimes V) \\
     \dTo^{\delta^{(i)}} & & \dTo^{\cong} \\
     H^i(L,V) & \rTo^{\exp^*_{(d-i),V^*(d),K}} &  H^{i-1}_{\dR}(V)\oplus H^{i}_{\dR}(V) 
    \end{diagram}
    The map $\delta^{(i)}$ will be defined in Section~\ref{EL}.
   \end{thm}
   
   The construction of the maps $\delta^{(i)}$ relies on the calculation of the Galois cohomology groups $H^i(L,V)$ in terms of the  $(\phi,G_L)$-module of $V$: if $\calC^\bullet(\DD(V))$ denotes the continuous cochain complex calculating the $G_L$-cohomology of $\DD(V)$, then by Theorem 3.3 in~\cite{andreattaiovita07} the absolute Galois cohomology of $V$ is isomorphic to the cohomology of the mapping cone $\calT_\phi^\bullet(\DD(V))$ of $\calC^\bullet(\DD(V))\rTo^{\phi-1} \calC^\bullet(\DD(V))$. If $L$ is a finite extension of $\QQ_p$, then $G_L$ is isomorphic to an open subgroup of $\ZZ_p$ and hence topologically generated by one element, which makes the mapping cone $\calT_\phi^\bullet(\DD(V))$ comparatively easy to handle (c.f.~\cite{herr98}). In the general case however, the group $G_L$ is a non-commutative $p$-adic Lie group of dimension $\dim L$. In Section~\ref{Koszulcohom} we calculate the $G_L$-cohomoloy of $\DD(V)$ via the Koszul complex, which enables us to prove the crucial result that the submodule $\DD(V)^{\psi=0}$ has trivial $G_L$-cohomology (c.f. Proposition~\ref{zerocohom}). As an application, we show that the Galois cohomology of $V$ is also be calculated via the complex $\calT_\psi^\bullet(\DD(V))$, where $\psi$ is a left inverse of $\phi$.
   \vs
   
   In a forthcoming paper, we shall generalise the results of this paper to the relative situation (c.f.~\cite{andreatta06}).

%++++++++++++++++++++++++++++++++++
   \vs
   
  \noindent {\it Acknowledgements.} I am very grateful to Guido Kings for his interest and for his invitation to Regensburg in Spring 2007 when the results of this paper were proven. I also thank James Cranch for answering my questions about the Koszul complex and David L\"offler for many useful discussions and for the careful reading of the manuscript. 
   
%++++++++++++++++++++++++++++++++++

   \subsection{Notation}
    
    \begin{itemize}
     \item If $A$ is perfect ring of characteristic $p$, denote by $W(A)$ its ring of Witt vectors. For $a\in A$, denote by $[a]\in W(A)$ its Teichm\"uller representative.
     
     \item For a field $F$, denote by $O_F$ its ring of integers. If $F$ is a local field of characteristic $0$, denote by $k_F$ its residue field. Also, denote by $\bar{F}$ a separable closure of $F$.
   
     \item For a local field $K$ of mixed characteristic $(0,p)$, denote by $\CC_K$ the completion of an algebraic closure of $K$. Also, denote by $\CC_p$ the completion of an algebraic closure of $\QQ_p$. 
   
     \item We denote by $v_p$ the $p$-adic valuation on $\bar{\QQ}_p$ and by $\mid\sim\mid_p$ the induced norm, normalised as $\mid p\mid_p=p^{-1}$.
   
     \item Let $G$ be a profinite group and $M$ a topological abelian group with a continuous $G$-action. By $H^\bullet(G,M)$ we always mean continuous group cohomology. Write $\calC^\bullet(G,M)$ for the continuous cochain complex with coefficients in $M$ (c.f. Section II.3 in~\cite{nsw}).
     
     \item Let $K$ be a field and $\bar{K}$ a separable closure. If $M$ is an abelian group with a continuous action of $\Gal(\bar{K}\slash K)$, denote by $H^i(K,M)$ the $i$th Galois cohomology group $H^i(\Gal(\bar{K}\slash K),M)$. 
     
     \item For a profinite group $G$, denote by $\Lambda(G)$ its Iwasawa algebra, which is defined as the inverse limit $\varprojlim\ZZ_p[G\slash U]$, where $U$ runs over all open normal subgroups of $G$. 
     
    \end{itemize}

%++++++++++++++++++++++++++++++++++
%++++++++++++++++++++++++++++++++++

  \section{The Iwasawa tower}\label{Iwasawa_tower}

   Let $L$ be a complete discrete valuation field of characteristic $0$ with residue field $k_L$ of characteristic $p$. Suppose that $[k_L^p:k_L]=p^{d-1}$. Let $X_1,\dots,X_{d-1}\in L$ be a lift of a $p$-basis $\overline{X_1},\dots,\overline{X_{d-1}}$ of $k_L$. Choose a subfield $K$ of $L$ with the same residue field $k_L$ in which $p$ is a uniformizer. Then the extension $L\slash K$ is finite of degree $e_L$, where $e_L$ is the absolute ramification index of $L$. Note that if $d>1$, then $K$ is not unique. 
   \vs
   
   For $n\geq 1$ define $K_{n}= K(\mun,X_1^{\frac{1}{p^n}},\dots,X_{d-1}^{\frac{1}{p^n}})$, and let $K_{\infty}=\bigcup_n K_{n}$. Let $L_n=LK_{n}$ and $L_\infty=LK_{\infty}$. Then $L_\infty$ contains the fields $L^{(0)}_{\infty}=L(\mu_{p^{\infty}})$ and $L^{(i)}_{\infty}=L^{(0)}_{\infty}\big(X_1^{\frac{1}{p^\infty}},\dots,X_i^{\frac{1}{p^\infty}}\big)$ for all $1\leq i\leq d-1$. Define the Galois groups $\Gamma_L=\Gal(L^{(0)}_{\infty}\slash L)$, $H^{(i)}_L=\Gal(L^{(i)}_{\infty}\slash L^{(i-1)}_{\infty})$ and $H_L=\Gal(L_\infty\slash L^{(0)}_{\infty})$. Also, let $G_L=\Gal(L_\infty\slash L)$, $\calG_L=\Gal(\bar{K}\slash L)$ and $\calH_L=\Gal(\bar{K}\slash L_\infty)$. We identify $\Gamma_L$ via the quotient map with the subgroup $\Gal(L_\infty\slash L'_\infty)$ of $G_L$, where 
   $L'_\infty=\varinjlim L(X_1^{\frac{1}{p^n}},\dots,X_{d-1}^{\frac{1}{p^n}})$.
   \vs
   
   \noindent {\bf Note.} (1) For all $1\leq i\leq d-1$, $H^{(i)}_L\cong\ZZ_p$ via some character $\eta_i$.
   
   \noindent (2) The cyclotomic character $\chi$ identifies $\Gamma_L$ with an open subgroup of $\ZZ_p^\times$.
   
   \begin{lem}
    We have $G_L\cong\Gamma_L\rtimes H_L$, so $G_L$ is a non-commutative $p$-adic Lie group of dimension $d$. Moreover, for all $1\leq i<d$, we have $H^{(i)}_L\cong \ZZ_p(1)$ as a $\Gamma$-module. 
   \end{lem}

   Let $\gamma$ be a topological generator of $\Gamma$, and for all $1\leq i\leq d$ let $h_i$ be a topological generator of $H_i$. 
         
   \begin{cor}
    For all $1\leq i\leq d-1$, we have
    \begin{equation*}
     \gamma h_i=h_i^{\chi(\gamma)}\gamma.
    \end{equation*}
   \end{cor}
   
   The following standard result will be important in Section~\ref{Koszulcohom}
   
   \begin{lem}
    The Iwasawa algebra $\Lambda(H_L)$ is isomorphic to the power series ring $\ZZ_p[[Y_1,\dots,Y_{d-1}]]$.
   \end{lem}
   
   \noindent {\bf Note.} If $x\in\Lambda(H_L)$, say $x=\sum_{I\in\ZZ^{d-1}} a_I{\bf Y}^I$ where for $I=(i_1,\dots,i_{d-1})$ we write ${\bf Y}^I$ for $Y_1^{i_1}\dots Y_{d-1}^{i_{d-1}}$, then $x$ is invertible in $\Lambda(H_L)$ if and only if $a_0\in\ZZ_p^\times$. 
   \vs
   
   We write $\frak{M}$ for the unique maximal ideal of $\Lambda(H_L)$.

%+++++++++++++++++++++

%++++++++++++++++++++++++++++++++++
%++++++++++++++++++++++++++++++++++

 \section{Theory of $p$-adic representations}
 
%+++++++++++++++++++++++++++++++

  \subsection{Fields of norms}\label{fields_of_norms}
  
   We follow the construction in~\cite{andreatta06} and~\cite{scholl06}. Then the tower $(K_n)_{n\geq 1}$ is strictly deeply ramified, so there exists $n_0\in\NN$ and $0<|\xi|_K<1$ such that for all $n\geq n_0$ the $p$-power map $O_{K_{n+1}}\slash\xi\rightarrow O_{K_n}\slash \xi$ is a surjection.
   \vs
    
   \noindent {\bf Definition.} Let $\EE^+_K=\varprojlim_{n\geq n_0} O_{K_n}\slash\xi$, where the inverse limit is taken with respect to the $p$-power maps.
   \vs
   
   Then $\EE_K^+$ is a complete discrete valuation ring of characteristic $p$ independent of $n_0$ and $\xi$. Let $\bar{\pi}$ be a uniformizer of $\EE_K^+$, and let $\EE_K=\EE_K^+[\bar{\pi}^{-1}]$ be the fraction field of $\EE_K^+$. We call $\EE_K$ the field of norms of the tower $(K_n)$. Note that $\EE_K$ is quipped with a natural action of $G_K$ which commutes with the Frobenius operator $\phi$. Let $v_{\EE}$ be the discrete valuation on $\EE_K$ normalised by $v_{\EE}(\bar{\pi})=1$. 

   One can show (c.f. Section 4.3 in~\cite{andreatta06}) that $\EE_K^+$ is a subring of the ring $\tilde{\EE}^+_K$ which is defined as follows: Let $\widehat{K_\infty}$ be the $p$-adic completion of $K_\infty$, and define $\tilde{\EE}_K^+$ to be the ring
   \begin{equation*}
    \tilde{\EE}_K^+=\big\{(x^{(0)},x^{(1)},\dots,x^{(i)},\dots)\mid x^{(i)}\in\widehat{K_\infty},\big(x^{(i+1)}\big)^p=x^{(i)}\big\},
   \end{equation*}
   where the transition maps are defined by raising to the $p$th power, the multiplicative structure is induced by the one on $\widehat{K_\infty}$ and the additive structure is defined as
   \begin{equation*}
    (\dots,x^{(i)},\dots)+(\dots,y^{(i)},\dots)=\big(\dots,\lim_{in\rightarrow +\infty}(x^{(i+n)}+y^{(i+n)})^{p^{i+n}},\dots\big).
   \end{equation*}
   Let $\tilde{\EE}^+=\varprojlim O_{\CC_K}$, where the inverse is given by the $p$-power map. Then $\tilde{\EE}^+$ is a complete valuation ring, with the ring structure defined as above, of characteristic $p$. Let $\tilde{\EE}$ be its field of fractions. Extend $v_{\EE}$ to a map $v_{\EE}:\tilde{\EE}\rightarrow \RR\cup\{+\infty\}$ by
   \begin{equation*}
    v_{\EE}(x)=\max\big\{n\in\QQ\mid x\in\bar{\pi}^{n}\tilde{\EE}^+\big\}.
   \end{equation*}
   For every finite extension $K'$ of $K$, $\EE_{K'}$ is a finite separable extension of $\EE_K$ of degree $[K'_\infty:K_\infty]$. Let $\EE=\bigcup_{K'}\EE_{K'}\subset\tilde{\EE}$. 
   
   \begin{thm}
    There is an isomorphism of toplogical groups
    \begin{equation*}
     \Gal(\EE\slash\EE_K)\cong \Gal(\bar{K}\slash K_\infty).
    \end{equation*}
   \end{thm}
   \begin{proof}
    Corollary 6.4 in~\cite{andreatta06}.
   \end{proof}
   
   For $K'$ a finite extension of $K$, let $\EE_{K'}=\EE^{\Gal(\bar{K}\slash K'_\infty)}$. Note that $\EE=\bigcup_{K'} \EE_{K'}$. 
   \vs
   
   For all $n\geq 1$ let $\varepsilon^{(n)}$ be a primitive $p^n$th root of unity such that $\big(\varepsilon^{(n)}\big)^p=\varepsilon^{(n-1)}$. Define  $\varepsilon=(1,\varepsilon^{(1)},\varepsilon^{(2)},\dots)$ and $x_i=(X_i,X_i^{\frac{1}{p}},\dots)$ for $1\leq i\leq d-1$, which are elements in $\EE_K^+$. One can show that $\varepsilon-1$ is a uniformizer of $\EE_K$. From now on, we let $\bar{\pi}=\varepsilon-1$. 
   
   \begin{prop}
    We have $\EE_K\cong k_K((\bar{\pi}))$. The isormorphism is given by identifying $x_i$ with the element $\overline{X_i}$. 
   \end{prop}
   \begin{proof}
    See Section 2.3 in~\cite{scholl06}.
   \end{proof}

%+++++++++++++++++++++++++++++++++++++++++++++++++++++++++++++++++++++++++++++++++++++++++++++++++++++++++++++++++++++++++++++++++++++++++++++++
%+++++++++++++++++++++++++++++++++++++++++++++++++++++++++++++++++++++++++++++++++++++++++++++++++++++++++++++++++++++++++++++++++++++++++++++++

  \subsection{$(\phi,G_K)$-modules}\label{eqcat}

   Define $\tilde{\AA}_K=W(\tilde{\EE}_K)$ (resp. $\tilde{\AA}$) to be the ring of Witt vectors of $\tilde{\EE}_K$ (resp. $\tilde{\EE}$). Then $\tilde{\AA}$ is equipped with two topologies - the strong $p$-adic topology, and the weak topology which is defined as follows: consider on $\tilde{\EE}$ the topology having $\{\bar{\pi}^n\tilde{\EE}^+\}$ as a fundamental system of neighbourhoods of $0$. On the truncated Witt vectors $W_m(\tilde{\EE})$ we consider the product topology via the isomorphism $W_m(\tilde{\EE})\cong (\tilde{\EE})^m$. The weak topology is then defined as the projective limit topology $W(\tilde{\EE})=\varprojlim_m W_m(\tilde{\EE})$. 
    
   \begin{prop}\label{base}
    There exists a subring $\AA_K$ of $\tilde{\AA}_K$ which is complete and separated for the weak topology and is stable under the actions of $G_K$ and of Frobenius $\phi$ and such that $\AA_K\slash p\AA_K\cong\EE_K$. 
   \end{prop}
   \begin{proof}
    Appendix C in~\cite{andreatta06}.
   \end{proof}
   
   Explicitly, $\AA_K$ can be constructed as follows: for $1\leq i\leq d-1$, let $T_i=[x_i]$ be the Teichm\"uller representative of $x_i$. Also, let $\pi=[\varepsilon]-1$. Then as shown in Section 2.3 in~\cite{scholl06}, the ring $\AA_K$ is isomorphic to the $p$-adic completion $O_K((\pi))^\wedge$ of $O_K((\pi))$ via identifying $X_i$ with $T_i$. Define $\AA_K^+=O_K[[\pi]]$.
   \vs
    
   \noindent {\bf Definition.} For a finite extension $K'$ of $K$, define $\AA_{K'}$ to be the unique $\AA_K$-algebra lifting the finite \'etale extension $\EE_{K'}$ over $\EE_K$. Let $\AA$ be the closure of $\bigcup_{K'} \AA_{K'}$ in $\tilde{\AA}$ for the $p$-adic topology.
   \vs
   
   \begin{prop}
    For every finite extension $K'$ of $K$, $\AA_{K'}$ is complete and separated for the weak topology.
   \end{prop}
   \begin{proof}
    Immediate from Proposition~\ref{base}.
   \end{proof}
   
   \begin{cor}\label{complete}
    As a toplogical group, $\AA_{K'}$ is the projective limit of additive discrete $p$-groups.
   \end{cor}
   \begin{proof}
    A basis of neighbourhoods of $0$ in the weak topology is given by $U_{m,n}=p^n\AA_{K'}+\pi^m\AA_{K'}^+$ for $m,n\geq 0$. Then $\AA_{K'}\Mod U_{m,n}$ has the discrete topology, and $\AA_{K'}=\varprojlim \AA_{K'}\slash U_{m,n}$.
   \end{proof}

   By construction, $\AA$ is equipped with an action of $\calH_K$ which commutes with $\phi$, and one can show (c.f. Proposition 7.9 in~\cite{andreatta06}) that for any finite extension $K'$ of $K$ one has $\AA^{\calH_{K'}}=\AA_{K'}$. Let $\BB=\AA[p^{-1}]$ and $\BB_{K'}=\AA_{K'}[p^{-1}]$.
   \vs
   
   \noindent {\bf Definition.} Let $\Rep(\calG_L)$ (resp. $\Rep_{\ZZ_p}(\calG_L)$) be the abelian tensor category of finite dimensional $\QQ_p$-vector spaces (resp. finitely generated $\ZZ_p$-modules) with a continuous action of $\calG_L$. Let $(\phi,G_L)-\MOD_{\BB_L}$ (resp. $(\phi,G_K)-\MOD_{\AA_L}$) be the abelian tensor category of finite dimensional vector spaces $D$ over $\BB_L$ (resp. finitely generated $\AA_L$-modules) equipped with 
   
   (i) a semi-linear action of $G_L$;
   
   (ii) a semi-linear action of a homomorphism $\phi$ commuting with $G_L$. 
   
   For a $p$-adic representation $V\in\Rep(\calG_L)$ (resp. $T\in\Rep_{\ZZ_p}(\calG_L)$), define 
   \begin{equation*}
    \DD(V)=\big(V\otimes_{\ZZ_p}\AA\big)^{\calH_L}  \text{\hspace{4ex} (resp. $\DD(T)=\big(T\otimes_{\ZZ_p}\AA\big)^{\calH_L}$)}.
   \end{equation*}
   Then $\DD(V)$ is a finite dimensional vector space over $\BB_L$ (resp. an \'etale finitely generated $\AA_K$-module) endowed with a semi-linear action of $G_L=\calG_L\slash\calH_L$. The homomorphism $\phi$ on $\AA$ defines a semi-linear action of $\phi$ on $\DD(V)$ commuting with the action of $G_L$. 
   
   For an object $D\in(\phi,G_L)-\MOD_{\BB_L}$, define
   \begin{equation*}
    \calV(D)=\big(\BB\otimes_{\BB_L}D\big)^{\phi=1} \text{\hspace{4ex} (resp. $\calV(D)=\big(\AA\otimes_{\AA_L}D\big)^{\phi=1}$)},
   \end{equation*}
   which is a finite-dimensional $\QQ_p$-vector space (resp. finitely generated $\ZZ_p$-modules) with a continuous action of $\calG_L$ induced from the action of $\calG_L$ on $\AA$ and the action of $G_L$ on $D$. 
   \vs
   
   \noindent {\bf Definition.} An element $D\in (\phi,G_L)-\MOD_{\AA_L}$ (resp. $D\in (\phi,G_L)-\MOD_{\BB_L}$) is \'etale if $D$ is generated by $\phi(D)$ as an $\AA_L$-module (resp. if $D$ contains an \'etale $\AA_L$-sublattice). Denote by $(\phi,G_L)-\MOD^{\et}_{\AA_L}$ (resp. $(\phi,G_L)-\MOD^{\et}_{\BB_L}$) the category of \'etale $(\phi,G_L)$-modules over $\AA_L$ (resp. $\BB_L$).
   
   \begin{thm}
    The functors $\DD$ and $\calV$ are inverse to each other and define an equivalence of abelian tensor categories between $\Rep(\calG_L)$ and $(\phi,G_L)-\MOD^{\et}_{\BB_L}$ (resp. between $\Rep_{\ZZ_p}(\calG_L)$ and $(\phi,G_L)-\MOD^{\et}_{\AA_L}$).
   \end{thm}
   \begin{proof}
    Theorem 7.11 in~\cite{andreatta06}.
   \end{proof}
   
   The following result will be important in Section~\ref{Koszulcohom}.
   
   \begin{prop}\label{projlim}
    Let $V\in\Rep_{\ZZ_p}(\calG_L)$. As a topological group, $\DD(V)$ is the projective limit of additive discrete $p$-groups.
   \end{prop}
   \begin{proof}
    The result follows from Corollary~\ref{complete} since $\DD(V)$ is a finitely generated $\AA_L$-module. 
   \end{proof}

%++++++++++++++++++++++++++++++++++++++++++++++++++++++++++++++++++++++++++++++++++++++++++++++++++++++++++++++++++++++++++ 
 
  \subsection{Rings of periods}
  
  \subsubsection{The ring $\BB_{\dR}$}
  
   Most of the results of this section are quoted from~\cite{kato99} and~\cite{brinon06}. Define $\theta:\tilde{\BB}^+\rightarrow \CC_K$ by
   \begin{equation*}
    \theta:\sum_{k\gg -\infty}p^k[z_k]\rightarrow \sum_{k\gg-\infty}p^kz_k^{(0)}.
   \end{equation*}
   Note that $\theta$ is surjective.
   
   \begin{lem}\label{principal}
    $\ker(\theta)$ is a principal ideal, generated by $\omega=\frac{[\varepsilon]-1}{[\varepsilon^{\frac{1}{p}}]-1}$.
   \end{lem}
   
   \noindent {\bf Definition.} Define $\BB_{\dR}^{\nabla +}$ to be the separated completion of $\tilde{\BB}^+$ for the $\ker(\theta)$-adic topology.
   \vs
   
   By Lemma~\ref{principal}, the ring $\BB_{\dR}^{\nabla +}$ is a discrete valuation ring with uniformizing element
   \begin{equation*}
    t=\log([\varepsilon])=\sum_{k=1}^{+\infty}(-1)^{k+1}\frac{\big([\varepsilon]-1\big)^k}{k}.
   \end{equation*}
   Define $\BB_{\dR}^\nabla=\BB_{\dR}^{\nabla +}[t^{-1}]$ to be the fraction field of $\BB_{\dR}^{\nabla +}$. Note that $\BB_{\dR}^\nabla$ is equipped with an action of $\calG_K$ and a separated and exhaustive decreasing filtration $\fil^\bullet$ defined by $\fil^r\BB_{\dR}^\nabla= t^r\BB_{\dR}^{\nabla +}$ for $r\in\ZZ$. 
   \vs
   
   The ring homomorphism $\theta:\tilde{\AA}^+\rightarrow O_{\CC_K}$ extends to a homomorphism $\theta_L:O_L\otimes_{\ZZ}\tilde{\AA}^+\rightarrow O_{\CC_K}$. 
   \vs
   
   \noindent {\bf Definition.} Let $\AA(O_{\CC_L}\slash O_K)$ be the separated completion of $O_L\otimes\tilde{\AA}^+$ for the topology defined by the ideal generated by $p$ and $\ker(\theta_L)$. Let $\theta_L:\AA(O_{\CC_K}\slash O_L)[p^{-1}]\rightarrow \CC_K$ denote the induced homomorphism. Define $\BB_{\dR}^+$ to be the separated completion of $\AA(O_{\CC_K}\slash O_L)[p^{-1}]$ for the $\ker(\theta_L)$-adic topology. Define $\BB_{\dR}=\BB_{\dR}^+[t^{-1}]$.
   \vs
   
   \noindent {\bf Notation.} For $1\leq i\leq d-1$, denote by $u_i$ the image of $X_i\otimes 1-1\otimes [x_i]\in O_L\otimes_{\ZZ} \tilde{\AA}^+$ in $\BB_{\dR}^+$.
   
   \begin{prop}\label{easyBdR}
    The homomorphism $\BB_{\dR}^{\nabla +}[[u_1,\dots,u_{d-1}]]\rightarrow \BB_{\dR}^+$ is an isomorphism of $\tilde{\AA}^+$-algebras.
   \end{prop}
   \begin{proof}
    See Proposition 2.1.7 in~\cite{brinon06}.
   \end{proof}

   \noindent {\it Action of $\calG_L$.} The continuous action of $\calG_L$ on $\BB_{\dR}^\nabla$ extends to a continuous action on $\BB_{\dR}$, and one can show the following results:
   
   \begin{lem}
    We have $H^0(L,\BB_{\dR})=L$ and $H^0(L,\BB^\nabla_{\dR})=F$, where $F$ is the maximal algebraic extension of $\QQ_p$ contained in $L$.
   \end{lem}
   \begin{proof}
    Proposition 2.1.13 in~\cite{brinon06}.
   \end{proof}
   
   A more general version of this result was proven by Kato~\cite{kato99} (c.f. Lemma~\ref{katoBdR}).
   \vs
   
   \noindent {\it Filtration on $\BB_{\dR}$.} Observe that $\BB_{\dR}^+$ is equipped with the filtration $\fil^\bullet$ defined by $\fil^r\BB_{\dR}^+=\ker(\theta_L)^r$ for $r\in\NN$. The ideal $\ker(\theta_L)$ is generated by $t,u_1,\dots,u_{d-1}$. Define a filtration $\Fil^{\bullet}$ on $\BB_{\dR}$ be putting
   \begin{align*}
    \Fil^0\BB_{\dR} & =\sum_{n=0}^\infty t^{-n}\fil^n\BB_{\dR}^+=\BB^+_{\dR}[t^{-1}u_1,\dots,t^{-1}u_{d-1}],\\
    \Fil^r\BB_{\dR} & =t^r\Fil^0\BB_{\dR} \text{\hspace{3ex} for $r\in\ZZ$}.
   \end{align*}
   One can show (c.f. Proposition 2.2.1 in~\cite{brinon06}) that the filtration $(\Fil^r\BB_{\dR})$ is decreasing, separated and exhaustive and stable under the action of $\calG_L$. 
   \vs
   
   \noindent {\it Connection on $\BB_{\dR}$.} By Proposition~\ref{easyBdR}, we know that $\BB_{\dR}^+=\BB_{\dR}^{\nabla +}[[u_1,\dots,u_{d-1}]]$. For $i\in\NN$ define 
   \begin{equation*}
    \Omega^i_{O_K}=\varprojlim \Omega^i_{O_K\slash\ZZ}\slash\ p^n\Omega^i_{O_K\slash\ZZ},
   \end{equation*}
   and let $\Omega^i_K=\QQ_p\otimes_{\ZZ_p}\Omega^i_{O_K}$. Note that the $K$-vector space $\Omega^1_K$ has the basis $\{d\log X_i\}_{1\leq i<d}$, so $\Omega^j_K=0$ for all $j\geq d$. For $1\leq i\leq d-1$, let $N_i$ be the unique continuous $\BB_{\dR}^{\nabla +}$-derivation on $\BB_{\dR}^+$ such that $N_i(u_j)=\delta_{ij}$. Then $N_i(t)=0$, so $N_i$ extends to a continuous $\BB_{\dR}^\nabla$-derivation on $\BB_{\dR}$. Define the connection $\nabla$ on $\BB_{\dR}$ by
   \begin{align*}
    \nabla:\BB_{\dR} & \rTo \BB_{\dR}\otimes_{K}\Omega_K^1 \\
    x & \rTo \sum_{i=1}^{d-1}N_i(x)\otimes d\log(X_i).
   \end{align*}
   
   \begin{lem}\label{kernelconnection}
    We have $\big(\BB_{\dR}^+\big)^{\nabla=0}=\BB_{\dR}^{\nabla +}$ and $\big(\BB_{\dR}\big)^{\nabla=0}=\BB_{\dR}^{\nabla}$.
   \end{lem}
   \begin{proof}
    Proposition 2.2.8 in~\cite{brinon06}.
   \end{proof}
   
   \begin{prop}\label{filteredconnection}
    For all $r\in\ZZ$, the connection $\nabla$ gives an exact sequence
    \begin{equation*}
     0\rightarrow \fil^r\BB_{\dR}^\nabla\rightarrow \Fil^r\BB_{\dR}\rightarrow \Fil^{r-1}\BB_{\dR}\otimes_K\Omega^1_K\rightarrow \dots\rightarrow \Fil^{r+1-d}\BB_{\dR}\otimes_K\Omega^{d-1}_K\rightarrow 0.
    \end{equation*}
   \end{prop}
   \begin{proof}
    Proposition (2.1.10) in~\cite{kato99}.
   \end{proof}
   
   \begin{cor}\label{directlimit}
    The connection $\nabla$ gives exact sequences
    \begin{align*}
     0& \rightarrow \BB_{\dR}^\nabla\rTo \BB_{\dR}\rTo^\nabla \BB_{\dR}\otimes_K\Omega^1_K\rTo^\nabla \dots\rTo^\nabla \BB_{\dR}\otimes_K\Omega^{d-1}_K\rightarrow 0.
    \end{align*}
   \end{cor}
   \begin{proof}
    Immediate from Proposition~\ref{filteredconnection} by passing to the direct limit over $r$.
   \end{proof}
   
   \noindent {\it de Rham representations.} Let $V$ be a $p$-adic representation of $\calG_L$. Define $\DD_{\dR}(V)=\big(V\otimes\BB_{\dR}\big)^{\calG_K}$ and $\DD_{\dR}^\nabla=\big(V\otimes\BB^\nabla_{\dR}\big)^{\calG_K}$. Then $\DD_{\dR}(V)$ (resp. $\DD_{\dR}^\nabla(V)$) is a finite dimensional vector space over $L$ (resp. over $F$) of dimension $\leq\dim_{\QQ_p}V$. 
   \vs
   
   \noindent {\bf Definition.} A $p$-adic representation $V$ of $\calG_L$ is de Rham if $\dim_L\DD_{\dR}(V)=\dim_{\QQ_p}V$.
   \vs
   
   The vector space $\DD_{\dR}(V)$ is equipped with a filtration $\Fil^\bullet$ and a connection $\nabla:\DD_{\dR}(V)\rightarrow \DD_{\dR}(V)\otimes_K\Omega^1_K$. Moreover, it follows from Lemma~\ref{kernelconnection} that $\DD_{\dR}^\nabla(V)=\big(\DD_{\dR}(V)\big)^{\nabla=0}$. 
   \vs
   
   \noindent {\bf Notation.} To simplify the notation, we will write $\DD_{\dR}^r(V)$ for $\Fil^r\DD_{\dR}(V)$. 
   
  \subsubsection{The ring $\BB_{\max}^\nabla$}
  
  Let $A_{\max}^\nabla$ be the separated completion in the $p$-adic topology of the sub-$\tilde{\AA}^+$-algebra of $\tilde{\AA}^+$ generated by $p^{-1}\ker(\theta)$. Note that $t\in A_{\max}^\nabla$.
  \vs
  
  \noindent {\bf Definition.}  Define $\BB^{\nabla +}_{\max}=A^\nabla_{\max}[p^{-1}]$ and $\BB_{\max}^\nabla=\BB_{\max}^{\nabla +}[t^{-1}]$. 
  \vs
   
  It is clear from the definition that $A_{\max}$ inherits from $\tilde{\AA}^+$ a continuous action of $\calG_K$ and of the Frobenius operator $\phi$, and it is easy to see that $\phi(t)=pt$ and $g(t)=\chi(g)t$. These actions therefore extend to actions of $\calG_K$ and $\phi$ on $\BB_{\max}^{\nabla}$.
  
  \begin{lem}
   We have $H^0(L,\BB_{\max}^\nabla)=F_0$, where $F_0$ is the maximal unramified extension of $\QQ_p$ contained in $F$.
  \end{lem}
  \begin{proof}
   Corollary 2.4.11 in~\cite{brinon03}.
  \end{proof}
  
  \noindent {\bf Definition.} For a $p$-adic representation $V$ of $\calG_L$, define $\DD_{\max}^\nabla(V)=\big(V\otimes_{\QQ_p}\BB_{\max}^\nabla\big)^{\calG_L}$, which is a finite dimensional $F_0$-vector space of dimension $\leq \dim_{\QQ_p}V$. 
  \vs
   
   The following result is crucial in the construction of the exponential maps. 
   
   \begin{prop}\label{SesBK}
    The natural inclusion $\BB_{\max}^{\nabla \phi=1}\rightarrow \BB_{\dR}^\nabla$ induces a short exact sequence of $\calG_L$-modules
    \begin{equation}\label{BKsequence}
      0\rTo \QQ_p\rTo \BB_{\max}^{\nabla \phi=1}\rTo \BB_{\dR}^\nabla\slash\BB_{\dR}^{\nabla +}\rTo 0.
    \end{equation}
   \end{prop}
   \begin{proof}
    Proposition 2.4.16 in~\cite{brinon06}.
   \end{proof}

%++++++++++++++++++++++++++++++++++++++++++++++++++++++++++++++++++++++++++++++++++++++++++++++++++++++++++++++++++++++
%++++++++++++++++++++++++++++++++++++++++++++++++++++++++++++++++++++++++++++++++++++++++++++++++++++++++++++++++++++++

 \section{Overconvergent $(\phi,G_K)$-modules}
  
%+++++++++++++++++++++++++++++++++++++++++++++++++++++++++++++++++++++++++++++++++++++++++++++++++++++++++++++++++++++++++++++++++
  
  \subsection{Construction}\label{construction_ov}
  
   The rings of overconvergent series were defined in~\cite{andreattabrinon06}. For $r\in\QQ_{>0}$, let $\tilde{\AA}^{(0,r]}$ be the set of elements $z=\sum_{k=0}^\infty p^k[z_k]$ of $\tilde{\AA}$ such that 
   \begin{equation*}
    \lim_{k\rightarrow +\infty} rv_{\EE}(z_k)+k=+\infty.
   \end{equation*}
   
   Note that if $r<r'$, then $\tilde{\AA}^{(0,r']}\subset \tilde{\AA}^{(0,r]}$. Let $\tilde{\BB}^{(0,r]}=\tilde{\AA}^{(0,r]}[p^{-1}]$, $\tilde{\AA}^\dagger=\bigcup_r\tilde{\AA}^{(0,r]}$ and $\tilde{\BB}^\dagger=\tilde{\AA}^\dagger[p^{-1}]$. Define $\AA^{(0,r]}=\tilde{\AA}^{(0,r]}\cap\AA$, and $\BB^{(0,r]}=\tilde{\BB}^{(0,r]}\cap\BB$, where the intersection is taken in $\tilde{\BB}$. Note that $\phi$ induces maps $\AA^{(0,r]}\rightarrow \AA^{(0,pr]}$ and $\BB^{(0,r]}\rightarrow \BB^{(0,pr]}$. Put $\AA^\dagger=\bigcup_r \AA^{(0,r]}=\tilde{\AA}^\dagger\cap\AA$ and $\BB^\dagger=\bigcup_r \BB^{(0,r]}=\tilde{\BB}^\dagger\cap\BB$. Finally, for a finite extension $L$ of $K$, put $\AA_L^{(0,r]}=\big(\AA^{(0,r]}\big)^{\calH_L}$, $\BB_L^{(0,r]}=\big(\BB^{(0,r]}\big)^{\calH_L}$, $\AA_L^\dagger=\big(\AA^\dagger\big)^{\calH_L}$ and $\BB_L^\dagger=\big(\BB^\dagger\big)^{\calH_L}$.
   
%   \begin{lem}
%    The ring $\AA_L^{(0,r]}$ is separated and complete for the topology defined by $w_r$.
%   \end{lem}
%   \begin{proof}
%    Proposition 4.35 in~\cite{andreattabrinon06}.
%   \end{proof}
   \vs
   
   Using the isomorphism $\AA_K\cong O_K((\pi))^\wedge$, we can give an explicit (geometric) description of $\AA_K^{(0,r]}$. For $r>1$, let $\calA_K^{(0,r]}$ be the algebra of power series $F(Z)=\sum_{k\in\ZZ}a_kZ^k$ with $a_k\in O_K$ which converges on the annulus $\big\{Z\in\CC_K:0<v_p(Z)\leq r\big\}$ and is bounded.
   
   \begin{prop}\label{convergence}
    Suppose that $r>1$. Then the map which to $F\in\calA_K^{(0,r]}$ associates $F(\pi)$ gives an isomorphism between $\calA_K^{(0,r]}$ and $\AA_K^{(0,r]}$. 
   \end{prop}
      
   To prove Proposition~\ref{convergence}, we follow the strategy in Section II in~\cite{cherbonniercolmez98} and start by proving the following result: For $z=\sum_{k=0}^\infty p^k[z_k]\in\tilde{\AA}$ and $k\in\ZZ$, define $v_{\EE}^{\leq k}(z)=\inf_{i\leq k}v_{\EE}(z_i)$.
   
   \begin{lem}
    Let $z=\sum_{k=0}^\infty p^k[z_k]\in\tilde{\AA}$. Then
   \begin{equation*}
    \lim_{k\rightarrow +\infty} rv_{\EE}(z_k)+k=+\infty \text{ if and only if } \lim_{k\rightarrow +\infty} rv^{\leq k}_{\EE}(z)+k=+\infty.
   \end{equation*}
   \end{lem}
   \begin{proof}
    Suppose that $\lim_{k\rightarrow +\infty} rv^{\leq k}_{\EE}(z)+k=+\infty$. It is clear from the definition that $rv_{\EE}^{\leq k}(z)+k\leq rv_{\EE}(z_k)+k$ for all $k$, so certainly $\lim_{k\rightarrow +\infty} rv_{\EE}(z_k)+k=+\infty$. 
    
    Conversely, suppose that $\lim_{k\rightarrow +\infty} rv_{\EE}(z_k)+k=+\infty$. For all $k$ let $i_k\leq k$ such that $v_{\EE}(z_{i_k})=v_{\EE}^{\leq k}(z)$. Then certainly $\lim_{k\rightarrow +\infty} rv_{\EE}(z_{i_k})+i_k=+\infty$. But this implies the result since $rv_{\EE}(z_{i_k})+i_k\leq rv_{\EE}(z_{i_k})+k$.
   \end{proof}
   
   \noindent As shown in the proof of Proposition 4.3 in~\cite{andreattabrinon06}, the map $v_{\EE}^{\leq k}$ has the following properties:
   
   (1) $v_{\EE}^{\leq k}(z)=\infty \Longleftrightarrow z\in p^{k+1}\tilde{\AA}$;
   
   (2)$v_{\EE}^{\leq k}(y+z)\geq \inf\big(v_{\EE}^{\leq k}(y),v_{\EE}^{\leq k}(z)\big)$ with equality if and only if $v_{\EE}^{\leq k}(y)\neq v_{\EE}^{\leq k}(z)$;
   
   (3) $v_{\EE}^{\leq k}(yz)\geq \inf_{i+j\leq k}\big(v_{\EE}^{\leq i}(y)+v_{\EE}^{\leq j}(z)\big)$;
   
   (4) $v_{\EE}^{\leq k}(\phi(z))=pv_{\EE}^{\leq k}(z)$;
   
   (5) $v_{\EE}^{\leq k}(g(z))=v_{\EE}^{\leq k}(z)$ for all $g\in\calG_K$. 
   \vs
   
   \noindent Note that Proposition~\ref{convergence} can be reformulated as follows.
   
   \begin{lem}\label{equivalentcond}
    Let $z=\sum_{i\in\ZZ}a_i\pi^i\in\AA_K$. Then the following two conditions are equivalent:
    
    (i) $z\in\AA^{(0,r]}_K$;
    
    (ii) $\lim_{k\rightarrow -\infty} rv_K(a_k)+k =+\infty$.
   \end{lem}
   \begin{proof}
    Note first that $z=\sum_{i\in\ZZ}a_i\pi^i\in\AA_K^{(0,r]}$ if and only if $z'=\sum_{i<0}a_i\pi^k\in\AA_K^{(0,r]}$.     For $k\in\NN$, define
    $y_k=p^{-k}\sum_{i<0,v_K(a_i)=k}a_i\pi^i$ and $i_k=\inf\{i<0\mid v_K(a_i)=k\}$. By Lemma II.2.3 in~\cite{cherbonniercolmez98}, $z'\in \AA_K^{(0,r]}$ if and only if $\lim_{j\rightarrow -\infty}v_{\EE}(y_j)+rj=+\infty$. It follows from the definition that $y_k=[\bar{\pi}]^{i_k}u$, where
    \begin{equation*}
     u= \frac{a_{i_k}}{p^k}\Big(\frac{\pi}{\bar{\pi}}\Big)^{i_k}\Big(1+\sum_{i>i_k,v_K(a_i)=v_K(a_{i_k})}\frac{a_i}{a_{i_k}}\pi^{i-i_k}\Big)
    \end{equation*}
    is a unit in $\tilde{\AA}$. We therefore deduce that $v_{\EE}^0(y_k)=i_k$ and $v_{\EE}^{\leq j}(y_k)\geq v_{\EE}(y_k)-j$ for all $j>1$. On the other hand, since $z'=\sum_{j=0}^{+\infty}p^jy_j$, we have $v_{\EE}^{\leq k}(z')\geq\inf_{0\leq j\leq k}v_{\EE}^{\leq k}(p^jy_j)=\inf_{0\leq j\leq k}v_{\EE}^{\leq k-j}(y_j)$ with equality if $v_{\EE}^{\leq k}(y_k)<v_{\EE}^{\leq k-j}(y_j)$. Now the formula $v_{\EE}^0(y_k)=i_k$ implies that the following two conditions are equivalent:
    
    (1) $\lim_{k\rightarrow +\infty}v_{\EE}^{\leq k}(y_k)+rk=+\infty$;
    
    (2) $\inf_{v_K(a_j)=k} rv_K(a_j)+j \rightarrow +\infty$ as $k\rightarrow +\infty$.
    
    \noindent Since there is only a finite number of $j<0$ such that $v_K(a_j)=k$ for every $k$, condition (2) is equivalent to the condition that $\lim_{k\rightarrow +\infty}rv_K(a_k)+k=+\infty$, which finishes the proof.
   \end{proof}

   Define $A=\AA_K^+\big[\big[\frac{p}{\pi^{p-1}}\big]\big]$. Let $x\in\AA_K$. For $n\in\NN$, let $\omega_n(x)$ be the smallest integer $k$ such that $x\in \pi^{-k}A+p^{n+1}\AA_K$. The following result will be useful in Section~\ref{psioperator}:

   \begin{lem}\label{convergence2}
    Let $x=\sum_{k\in\ZZ}a_k\pi^k\in\AA_K$ for some $a_k\in O_K$. Then $x\in\AA_K^{(0,r]}$ if and only if  $\omega_n(x)-n\big(\frac{r}{p-1}-1\big)\rightarrow -\infty$ as $n\rightarrow +\infty$. 
   \end{lem}
   \begin{proof}
    Consequence of Proposition~\ref{convergence} and the definition of $\omega_n$.
%    Assume that $\omega_n(x)-n(p-1)(p^{r-1}-1)\rightarrow -\infty$ as $n\rightarrow +\infty$.
   \end{proof}
   
   Let $V\in\Rep(\calG_L)$ (resp. $T\in\Rep_{\ZZ_p}(\calG_L)$), and define
   \begin{equation*}
    \DD^\dagger(V)=\big(V\otimes_{\ZZ_p}\AA^\dagger\big)^{\calH_L} \text{\hspace{4ex} (resp. $\DD^\dagger(T)=
    \big(T\otimes_{\ZZ_p}\AA^\dagger\big)^{\calH_L}$)},
   \end{equation*}
   so $\DD^\dagger(V)$ (resp. $\DD^\dagger(T)$) is a finitely generated module over $\BB^\dagger_L$ (resp $\AA_L^\dagger$) endowed with commuting semi-linear actions of $\phi$ and $G_L$. Also, define $\DD^{(0,r]}(V)=\big(V\otimes_{\ZZ_p}\AA^{(0,r]}\big)^{\calH_L}$. 
   
   \begin{thm}
    For any $V\in\Rep(\calG_L)$ (resp. $T\in\Rep_{\ZZ_p}(\calG_L)$), we have 
    \begin{equation*}
     \DD(V)=\BB_L\otimes_{\BB_L^\dagger}\DD^\dagger(V) \text{\hspace{4ex} (resp. $\DD(T)=\AA_L\otimes_{\AA_L^\dagger}\DD^\dagger(T)$)},
    \end{equation*}
    so $\DD(V)$ (resp. $\DD(T)$) has a basis of overconvergent elements.
   \end{thm}
   \begin{proof}
    Theorem 4.42 in~\cite{andreattabrinon06}.
   \end{proof}
   
   \noindent {\bf Note.} Since $V$ is finite-dimensional, there exists $r_V> 0$ such that $\DD(V)=\BB_L\otimes_{\BB_L^{(0,r_V]}}\DD^{(0,r_V]}(V)$. \vs
   
   The rings $\BB^{(0,r]}$ are important since they are the bridge between the world of $(\phi,G_L)$-modules and the world of de Rham representations.
   
   \begin{lem}\label{convergenceBdR}
    An element $\sum_{k=0}^\infty p^k[z_k]$ converges in $\BB_{\dR}^{\nabla +}$ if and only if $\sum_{k=0}^\infty p^kz_k^{(0)}$ converges in $\CC_K$.
   \end{lem}
   \begin{proof}
    If $\sum_{k=0}^\infty p^k[z_k]$ converges in $\BB_{\dR}^{\nabla +}$, then $\sum_{k=0}^\infty p^kz_k^{(0)}=\theta\Big(\sum_{k=0}^\infty p^k[z_k]\Big)$ is well-defined. Conversely, suppose that $\sum_{k=0}^\infty p^kz_k^{(0)}$ converges in $\CC_K$. If $a_k$ denotes the integer part of $v_{\EE}(z_k)$, then $a_k+k\rightarrow +\infty$ as $k\rightarrow +\infty$. We can therefore write $p^k[z_k]$ as $p^{a_k+k}u_k(p^{-1}[\tilde{p}])^{a_k}$, where $u_k\in\tilde{\AA}^+$ and $\tilde{p}\in\tilde{\EE}^+$ satisfies $\tilde{p}^{(0)}=p$. Note that $p^{-1}[\tilde{p}]\in 1+\ker(\theta)$. Since $\BB_{\dR}^{\nabla +}$ is complete in the $(p,\ker(\theta))$-adic topology, this finishes the proof.
   \end{proof}
   
   \begin{prop}
    Let $n_r\in\NN$ be such that $p^{-n_r}<r$. Then for all $n\geq n_r$ the map $\phi^{-n}$ defines an injection
    \begin{equation*}
     \phi^{-n}: \tilde{\BB}^{(0,r]}\rTo \BB_{\dR}^{\nabla +}.
    \end{equation*}
   \end{prop}
   \begin{proof}
    Since $\tilde{\BB}^{(0,r]}=\tilde{\AA}^{(0,r]}[p^{-1}]$, it is sufficient to show that $\phi^{-n}$ defines an injection $\tilde{\AA}^{(0,r]}\rightarrow \BB_{\dR}^{\nabla +}$. Let $x=\sum_{k=0}^{\infty}p^k[z_k]\in\tilde{\AA}^{(0,r]}$, so by definition we have
    $\lim_{k\rightarrow +\infty} rv_{\EE}(z_k)+k=+\infty$, which implies that 
    \begin{equation}\label{convergence*}
     \lim_{k\rightarrow +\infty} rp^{-n}v_{\EE}(z_k)+k=+\infty.
    \end{equation}
    Now $\phi^{-n}(x)=\sum_{k=0}^\infty p^k[\phi^{-n}(z_k)]$, and $v_{\EE}(\phi^{-n}(z_k))=p^{-n}v_{\EE}(z_K)$, so $\sum_{k=0}^{\infty}p^kz_K^{(n)}$ converges in $\CC_K$ by~\eqref{convergence*}, which finishes the proof by Lemma~\ref{convergenceBdR}.
   \end{proof}
   
   \begin{cor}
    Let $V\in\Rep(\calG_L)$, and let $n\in\NN$ such that $p^{-n}<r_V$. Then $\phi^{-n}$ defines an injection
    \begin{equation*}
     \phi^{-n}:\DD^{(0,r]}(V)\rTo \big(V\otimes\BB_{\dR}^\nabla\big)^{\calH_L}.
    \end{equation*}
   \end{cor}

%+++++++++++++++++++++++++++++++++++++++++++++++++++++++++++++++++++++++++++++++++++++++++++++++++++++++++++++++++++++++++++++++++

  \subsection{The $\psi$-operator}\label{psioperator}
  
   We have seen in Section~\ref{eqcat} that $\phi$ defines a map $\BB\rightarrow \BB$. In fact, $\BB$ is a finitely generated vector space over $\phi(\BB)$ of dimension $p^d$. Explicitly, a basis of $\BB$ over $\phi(\BB)$ is given by $\Big\{(1+\pi)^{i_0}T_1^{i_1}\dots T_{d-1}^{i_{d-1}}\Big\}$,
   where $0\leq i_0,\dots,i_{d-1}\leq p-1$. Define
   \begin{align*}
    \psi: \BB&\rTo\BB \\
    x&\mapsto\frac{1}{p^d}\phi^{-1}\big(\Tr_{\BB\slash\phi(\BB)}(x)\big).
   \end{align*}
   The following properties of $\psi$ are immediate:
   
   (i) $\psi\circ\phi=\id$;
   
   (ii) $\psi$ commutes with the action of $\calG_L$;
   
   (iii) $\psi\big(\BB^\dagger\big)\subset\BB^{\dagger}$; more precisely, $\psi\big(\BB^{(0,r]}\big)\subset \BB^{(0,pr]}$ for all $r>0$. 
   
   Since $\psi$ commutes with the action of $\calG_L$, for any $V\in\Rep(\calG_L)$ the modules $\DD(V)$ and $\DD^\dagger(V)$ inherit an action of $\psi$ which commutes with the action of $G_L$. 
   \vs
   
   The main result of this section is Proposition~\ref{subset} below. In order prove it, we closely follow the strategy of the proof of Proposition III.3.2 in~\cite{cherbonniercolmez99}. We need a couple of preliminary results, whose proofs are identical to the $1$-dimensional case (c.f.~\cite{cherbonniercolmez99}). As in Section~\ref{construction}, define $A=\AA_K^+\big[\big[\frac{p}{\pi^{p-1}}\big]\big]$. Let $x\in\AA_K$. For $n\in\NN$, let $\omega_n(x)$ be the smallest integer $k$ such that $x\in \pi^{-k}A+p^{n+1}\AA_K$. It is easy to check that $\omega_n$ has the following properties:
   
   (1) $\omega_n(x+y)\leq \sup\{\omega_n(x),\omega_n(y)\}$;
   
   (2) $\omega_n(xy)\leq \sup_{i+j=n}\{\omega_i(x)+\omega_j(y)\}\leq \omega_n(x)+\omega_n(y)$;
   
   (3) $\omega_n(\phi(x))\leq p\omega_n(x)$.
   
   \begin{lem}
    (i) If $k\in\NN$, then $\phi(\pi^k)\in\AA_K^+$ and $\psi(\pi^{-k})\in\pi^{-k}\AA_K^+$.
    
    (ii) $\psi(A)\subset A$.
   \end{lem}
   \begin{proof}
    See Lemma I.6.1 in~\cite{cherbonniercolmez98}.
   \end{proof}
   
   \begin{cor}
    If $x\in\AA_K$ and $n\in\NN$, then $\omega_n(\psi(x))\leq 1+ \big[\frac{\omega_n(x)}{p}\big]\leq 1+\frac{\omega_n(x)}{p}$.
   \end{cor}
   \begin{proof}
    See Corollary I.6.3 in~\cite{cherbonniercolmez98}.
   \end{proof}
  
   \begin{prop}\label{subset}
    For any $V\in\Rep(\calG_L)$ or $\Rep_{\ZZ_p}(\calG_L)$ there exists $r_V'>0$ such that 
    \begin{equation*}
     \DD(V)^{\psi=1}\subset\DD^{(0,r_V']}(V).
    \end{equation*}
   \end{prop}
   \begin{proof}
    Let $T$ be a $\calG_L$-stable sublattice of $V$ (which is equal to $V$ if $V\in\Rep_{\ZZ_p}(V)$). Since $V$ is overconvergent, there exists a basis $e_1,\dots,e_m$ of $\DD(T)$ over $\AA_K$ such that $e_i\in V\otimes\AA^{(0,r]}$ for some $r> 0$. If $\Phi=(a_{i,j})$ denotes the matrix defined by $e_j=\sum_{i=1}^ma_{ij}\phi(e_j)$, then the entries of $\Phi$ are elements of $\AA_K^{(0,r]}$. If we define $\omega_k(\phi)$ to be the supremum of $\omega_k(a_{i,j})$, then 
    \begin{equation}\label{conditionoverconvergence}
     \omega_k(\Phi)-k\Big(\frac{r}{p-1}-1\Big)\rTo -\infty \text{\hspace{5ex} as $k\rightarrow +\infty$}
    \end{equation}
    by Lemma~\ref{convergence}. Let $y\in \DD^\dagger(V)^{\psi=1}$, and write $y=\sum_{i=1}^my_i\phi(e_i)$. To prove the proposition it is sufficient to show that $y_i\in\AA_L^{(0,r]}$ for all $i$. The relation $y=\psi(y)$ then translates as $y_i=\sum_{j=1}^ma_{ij}\psi(y_j)$,
    which implies that 
    \begin{align*}
     \omega_n(y_i) & \leq \sup_{1\leq j\leq m}\{\omega_n(a_{ij})+\omega_n(\psi(y_j))\}\\
                   & \leq \omega_n(\Phi)+\frac{\omega_n(y)}{p}+1
    \end{align*}
    for $1\leq i\leq m$. We deduce that $\omega_n(y)\leq \omega_n(\Phi)+\frac{\omega_n(y)}{p}+1$ and hence $\omega_n(y)\leq \frac{p}{p-1}(\omega_n(\Phi)+1)$. It follows therefore from~\eqref{conditionoverconvergence} that 
    \begin{equation*}
     \omega_k(y)-k\Big(\frac{r_V'}{p-1}-1\Big)\rTo -\infty \text{\hspace{5ex} as $k\rightarrow +\infty$},
    \end{equation*}
    where $r_V'=\frac{p}{p-1}r-1$.
   \end{proof}

%+++++++++++++++++++++++++++++++++++++++++++++++++++++++++++++++++++++++++++++++++++++++++++++++++++++++++++++++++++++++++++++++++
%+++++++++++++++++++++++++++++++++++++++++++++++++++++++++++++++++++++++++++++++++++++++++++++++++++++++++++++++++++++++++++++++++

 \section{The exponential maps and their duals}\label{exponentials+duals}
 
%++++++++++++++++++++++++++++++++++++++++++++++++++++++++++++++++ 
  
  \subsection{Construction}\label{construction}
  
   To simplify the notation, let $\BB^r_{\dR}=\Fil^r\BB_{\dR}$. The construction of the higher exponential maps relies on the following result of Kato.

   \begin{prop}\label{katoBdR}
    Let $V$ be a de Rham representation of $\calG_L$, and let $k\leq l$. Then 
    \begin{equation*}
     H^i(L,V\otimes_{\QQ_p}\BB^k_{\dR}\slash \BB^l_{\dR})=
     \begin{cases}
      \DD^k_{\dR}(V)\slash\DD^l_{\dR}(V) & \text{if $i=0,1$} \\
      0            & \text{otherwise}
     \end{cases}
    \end{equation*}
    Moreover, the isomorphism $\DD^k_{\dR}(V)\slash\DD^l_{\dR}(V)\cong H^1(L,V\otimes\BB^k_{\dR}\slash\BB^l_{\dR})$ is given 
    by taking cup product with $\log \chi$.
   \end{prop}
   \begin{proof}
    Proposition (2.3.3) in~\cite{kato99}.
   \end{proof}
      
   \begin{cor}
    Let $V$ be a de Rham representation of $\calG_L$. Then 
    \begin{equation*}
     H^i(L,V\otimes_{\QQ_p}\BB_{\dR})=
     \begin{cases}
      \DD_{\dR}(V) & \text{if $i=0,1$} \\
      0            & \text{otherwise}
     \end{cases}
    \end{equation*}
    Moreover, the isomorphism $\DD_{\dR}(V)\cong H^1(L,V\otimes\BB_{\dR})$ is given by taking cup product with $\log \chi$.
   \end{cor}
   \begin{proof}
    Since $\BB_{\dR}=\varinjlim t^{-k}\BB^0_{\dR}$, it is sufficient to show that the result holds when $\BB_{\dR}$ is replaced by $\BB_{\dR}^0$. But $\BB^0_{\dR}=\varprojlim \BB^0_{\dR}\slash\BB^l_{\dR}$, so the result follows from Lemma~\ref{katoBdR}, observing that the topology on $\BB_{\dR}^0$ is weaker than the $t$-adic topology.
   \end{proof}
   
   \noindent {\bf Definition.} By Corollary~\ref{directlimit} the connection $\nabla$ on $\BB_{\dR}$ induces an exact sequence of $\calG_L$-modules
   \begin{equation}\label{dualBK}
    0\rightarrow \BB_{\dR}^{\nabla}\rightarrow \BB_{\dR} \rTo^{\nabla} \BB_{\dR}\otimes_K\Omega^1_K
     \rTo^{\nabla} \dots\rTo^{\nabla} \BB_{\dR}\otimes_K\Omega^{d-1}_K\rightarrow 0,
   \end{equation}
   which gives rise to a spectral sequence $E_1^{m,n}\Rightarrow H^{m+n}(L,\BB_{\dR}^\nabla\otimes V)$, where $E_1^{m,n}=H^n(L,V\otimes_{\QQ_p}\BB_{\dR}\otimes_K\Omega^m_K)$.

   \begin{cor}\label{spectralseq}
    The spectral sequence $E_1^{m,n}\Rightarrow H^{m+n}(L,\BB_{\dR}^\nabla\otimes V)$ degenerates at $E_2$. 
   \end{cor}
   \begin{proof}
    Immediate since $H^i(L,V\otimes_{\QQ_p}\BB_{\dR}\otimes_K\Omega^m_K)=0$ for all $i\geq 2$ and $m\geq 0$ by Lemma~\ref{katoBdR}. 
   \end{proof}
    
   \noindent {\bf Definition.} For a $p$-adic representation $V$, define the de Rham cohomology $H^\bullet_{\dR}(V)$ of $V$ to be the cohomology of the complex
   \begin{equation*}
    0\rTo \DD_{\dR}(V)\rTo^\nabla\DD_{\dR}(V)\otimes_K\Omega^1_K
    \rTo^\nabla\dots\rTo^\nabla \DD_{\dR}(V)\otimes_K\Omega^{d-1}_K\rTo 0.
   \end{equation*}
   Note that for all $i\geq 0$, $H^i_{\dR}(V)$ is a filtered $L$-vector space.
   \vs
   
   \noindent {\bf Notation.} For all $1\leq i< d-1$, define
   \begin{align*}
    X^i_V &= \image\big(\nabla: V\otimes_{\QQ_p}\BB_{\dR}\otimes_K\Omega_K^{i-1}\rightarrow V\otimes_{\QQ_p}\BB_{\dR}\otimes_K\Omega_K^{i}\big)\\
          &= \ker\big(\nabla: V\otimes_{\QQ_p}\BB_{\dR}\otimes_K\Omega_K^i\rightarrow V\otimes_{\QQ_p}\BB_{\dR}\otimes_K\Omega_K^{i+1}\big).
   \end{align*}
   
   \noindent {\bf Remark.} For all $1\leq i\leq d-1$ we have a short exact sequence of $\calG_L$-modules
   \begin{equation}\label{Xi}
    0\rTo X^{i-1}_V\rTo V\otimes_{\QQ_p}\BB_{\dR}\otimes_K\Omega^{i-1}_K\rTo^\nabla X^i_V\rTo 0.
   \end{equation}
   
   \begin{lem}\label{quotient}
    For all $2\leq i\leq d$, we have an isomorphism
    \begin{equation*}
     \coker\big(\nabla: H^1(L,V\otimes_{\QQ_p}\BB_{\dR}\otimes_K\Omega_K^{i-2})\rightarrow H^1(L,X^{i-1}_V)\big)\cong H^i(L,V\otimes_{\QQ_p}\BB_{\dR}^\nabla).
    \end{equation*}
   \end{lem}
   \begin{proof}
    Clear from Corollary~\ref{spectralseq}.
   \end{proof}
      
   \begin{lem}\label{maps}
    For all $1\leq i\leq d$, the spectral sequence gives rise to maps
    \begin{align*}
     f^i_V&:H^i_{\dR}(V)\rightarrow H^i(L,V\otimes_{\QQ_p}\BB_{\dR}^\nabla),\\
     g^i_V&:H^i(L,V\otimes_{\QQ_p}\BB_{\dR}^\nabla)\rightarrow H^{i-1}_{\dR}(V). 
    \end{align*}
   \end{lem}
   \begin{proof}
    By Lemma~\ref{quotient}, it is sufficient to construct maps
    \begin{align*}
     \tilde{f}_V^i:& H^i_{\dR}(V)\rightarrow H^1(L,X^{i-1}_V),\\
     \tilde{g}^i_V:& H^1(L,X^{i-1}_V)\rightarrow H^{i-1}_{\dR}(V),
    \end{align*}
    such that $\tilde{g}^i_V$ factors through the image of $H^1(K,V\otimes_{\QQ_p}\BB_{\dR}\otimes_K\Omega_K^{i-2})$ (under $\nabla$). 
    \vs
     
    To construct the maps for $i=1$, observe that we have a short exact sequence
    \begin{equation*}
     0\rightarrow V\otimes_{\QQ_p}\BB_{\dR}^\nabla\rightarrow V\otimes_{\QQ_p}\BB_{\dR}\rightarrow X^1_V\rightarrow 0.
    \end{equation*}
    Taking $\calG_L$-cohomology defines maps $f^1_V:H^1_{\dR}(V)\rightarrow H^1(L,V\otimes_{\QQ_p}\BB_{\dR}^\nabla)$ and  $g^1_V:H^1(L,V\otimes_{\QQ_p}\BB_{\dR}^\nabla)\rightarrow H^1(L,V\otimes_{\QQ_p}\BB_{\dR})\cong\DD_{\dR}(V)$. Since $\nabla(g(x))=0$, it follows that in fact $g^1(x)\in \DD_{\dR}^\nabla(V)$ for all $x\in H^1(L,V\otimes_{\QQ_p}\BB_{\dR}^\nabla)$.
    \vs    
    
    In the general case, define $\tilde{f}^i_V$ to be the connection map $H^i_{\dR}(V)\rightarrow H^1(L,X^{i-1}_V)$ obtained by taking $\calG_L$-cohomology of~\eqref{Xi}. To construct $\tilde{g}^i_V$, note that the inclusion $X^{i-1}_V\rightarrow V\otimes_{\QQ_p}\BB_{\dR}\otimes_K\Omega_K^{i-1}$ induces a map $\frak{g}^i_V:H^1(L,X^{i-1}_V)\rightarrow H^1(L,V\otimes_{\QQ_p}\BB_{\dR}\otimes_K\Omega_K^{i-1})$. By Lemma~\ref{katoBdR}, $\frak{g}^i_V$ can be interpreted as a map $H^1(L,X^{i-1}_V)\rightarrow \DD_{\dR}(V)\otimes_K\Omega_K^{i-1}$. More precisely, we have a commutative diagram
    \begin{diagram}
     H^1(L, V\otimes_{\QQ_p}\BB_{\dR}\otimes_K\Omega_K^{i-1}) & \rTo^{\nabla} & H^1(L, X^i_V) & \rTo & H^1(L,V\otimes_{\QQ_p}\BB_{\dR}\otimes_K\Omega_K^{i}) \\
     \dTo^{\cong} & & & & \dTo^{\cong} \\
     \DD_{\dR}(V)\otimes_K\Omega_K^{i-1} & & \rTo^{\nabla} & & \DD_{\dR}(V)\otimes_K\Omega_K^i
    \end{diagram}
    The exactness of~\eqref{Xi} implies that in fact $\frak{g}^i_V(x)\in H^0(L,X^{i-1}_V)\subset \DD_{\dR}(V)\otimes_K\Omega_K^{i-1}$. Define $\tilde{g}^i_V(x)$ to be the image of $\frak{g}^i_V(x)$ under the natural quotient map $H^0(L,X^{i-1}_V)\rightarrow H^{i-1}_{\dR}(V)$. To see that $\tilde{g}^i_V$ factors though the image of $H^1(L,V\otimes_{\QQ_p}\BB_{\dR}\otimes_K \Omega_K^{i-2})$, it is now sufficient to observe that the image of $H^1(L,V\otimes_{\QQ_p}\BB_{\dR}\otimes_K \Omega_K^{i-2})$ in $\DD_{\dR}(V)\otimes_K\Omega_K^{i-1}$ is equal to $\image\big(\nabla: \DD_{\dR}(V)\otimes_K\Omega_K^{i-2}\rightarrow \DD_{\dR}(V)\otimes_K\Omega_K^{i-1}\big)$.
   \end{proof}
   
   \begin{lem}
    The map $f^i_V:H^i_{\dR}(V)\rightarrow H^i(L,V\otimes_{\QQ_p}\BB_{\dR}^\nabla)$ gives rise to an injection
    \begin{equation*}
     f^i_V:H^i_{\dR}(V)\rTo H^i(L,V\otimes_{\QQ_p}\BB_{\dR}^\nabla).
    \end{equation*}
   \end{lem}
   \begin{proof}
    By Lemma~\ref{quotient}, it is sufficient to show that if $y=f^i_V(x)$ is in the image of $H^1(L,V\otimes_{\QQ_p}\BB_{\dR}\otimes_K\Omega_K^{i-2})$ under $\nabla$, then $x=0$. Taking $\calG_L$-cohomology of~\eqref{Xi} shows that the image of $f^i_V(x)$ in $H^1(L,V\otimes_{\QQ_p}\BB_{\dR}\otimes_K\Omega^{i-1}_K)$ is zero. Now $H^1(L,V\otimes_{\QQ_p}\BB_{\dR}\otimes_K\Omega_K^{i-2})\cong \DD_{\dR}(V)\otimes_K\Omega^{i-2}_K$, so $y$ can be represented by a cocycle of the form $g\rightarrow \log\chi(g)\nabla(z)$ for some $z\in \DD_{\dR}(V)\otimes_K\Omega_K^{i-2}$. We have a commutative diagram
    \begin{diagram}
     H^1(L, V\otimes_{\QQ_p}\BB_{\dR}\otimes_K\Omega_K^{i-2}) & \rTo^{\nabla} & H^1(L, X^{i-1}_V) & \rTo & H^1(L,V\otimes_{\QQ_p}\BB_{\dR}\otimes_K\Omega_K^{i-1}) \\
     \dTo^{\cong} & & & & \dTo^{\cong} \\
     \DD_{\dR}(V)\otimes_K\Omega_K^{i-2} & & \rTo^{\nabla} & & \DD_{\dR}(V)\otimes_K\Omega_K^{i-1}
    \end{diagram}
    which implies that the image of $y$ in $H^1(L,V\otimes_{\QQ_p}\BB_{\dR}\otimes_K\Omega_K^{i-1})$ is trivial if and only if $\nabla(z)=0$ and hence $y=0$. Since $f^i_V$ gives an injection $H^i_{\dR}(V)\rightarrow H^1(L,X^{i-1}_V)$ by construction, this implies that $x=0$, which finishes the proof.
   \end{proof}

   \begin{prop}
    For all $1\leq i\leq d$, we have a split short exact sequence
    \begin{equation*}
     0\rightarrow H^i_{\dR}(V) \rTo^{f^i_V} H^i(L,V\otimes_{\QQ_p}\BB_{\dR}^\nabla)\rTo^{g^i_V} H^{i-1}_{\dR}(V)\rightarrow 0.
    \end{equation*}
   \end{prop}
   \begin{proof}
    We have already seen that $\image(g^i_V)\subset H^{i-1}_{\dR}(V)$. In order to prove the proposition it is therefore sufficient to construct a map $h^i_V:H^{i-1}_{\dR}(V)\rightarrow H^i(L,V\otimes_{\QQ_p}\BB_{\dR}^\nabla)$ such that $g^i_V\circ h^i_V=\id$. 
    \vs
    
    Define the map $\frak{h}^i_V:H^0(L,X^{i-1}_V)\rightarrow H^i(L,V\otimes_{\QQ_p}\BB_{\dR}^\nabla)$ by follows: for an element $x\in H^0(L,X^{i-1}_V)$, let $\frak{h}^i_V(x)$ be the image in $H^i(L,V\otimes_{\QQ_p}\BB_{\dR}^\nabla)$ of the cohomology class $[\sigma\rightarrow x\log\chi(\sigma)]$ under the surjective map $H^1(L,X^{i-1}_V)\cong H^i(L,V\otimes_{\QQ_p}\BB_{\dR}^\nabla)$. It is clear from Lemma~\ref{quotient} that $\frak{h}^i_V$ factors through
    \begin{equation*}
     h^i_V:H^{i-1}_{\dR}(V)\rightarrow H^i(L,V\otimes_{\QQ_p}\BB_{\dR}^\nabla)
    \end{equation*}
    and that $g^i_V\circ h^i_V=\id$, which finishes the proof.
   \end{proof}
      
   \begin{cor}
    For all $1\leq i\leq d$, $\nu^i_V=f^i_V\oplus h^i_V$ gives a decomposition of $H^i(L,V\otimes_{\QQ_p}\BB_{\dR}^\nabla)$ as
    \begin{equation*}
      H^i_{\dR}(V)\oplus H^{i-1}_{\dR}(V)\cong H^i(L,V\otimes_{\QQ_p}\BB_{\dR}^\nabla). 
    \end{equation*}
   \end{cor}
   
   \noindent {\bf Note.} We can consider this decomposition as a {\it Hodge decomposition}. By construction, $H^i_{\dR}(V)$ is a subquotient of $H^0(L,V\otimes_{\QQ_p}\BB_{\dR}\otimes_K\Omega_K^i)$, and $H^{i-1}_{\dR}(V)$ arises (under the isomorphism of Lemma~\ref{katoBdR}) as a subquotient of $H^1(L,V\otimes_{\QQ_p}\BB_{\dR}\otimes_K\Omega_K^{i-1})$.
   \vs
   
   Tensoring~\eqref{BKsequence} with $V$ and taking $\calG_L$-cohomology gives rise to connection homomorphisms $H^i\big(L,V\otimes_{\QQ_p} \BB_{\dR}^\nabla\slash  \BB_{\dR}^{\nabla +}\big)\rightarrow H^{i+1}(L,V)$ for all $i\geq 0$. Composing them with the natural maps $H^i(L,V\otimes_{\QQ_p}\BB_{\dR}^\nabla)\rightarrow H^i(L,V\otimes_{\QQ_p}\BB_{\dR}^\nabla\slash \BB_{\dR}^{\nabla +})$ gives maps $\delta_i:H^i(L,V\otimes_{\QQ_p}\BB_{\dR}^\nabla)\rTo H^{i+1}(L,V)$.
   \vs
      
   \noindent {\bf Definition.} For $1\leq i\leq d$, define the higher exponential map
   \begin{equation*}
    \exp_{(i),L,V}:H^i_{\dR}(V)\oplus H^{i-1}_{\dR}(V)\rightarrow H^{i+1}(L,V)
   \end{equation*} 
   to be the composition $\delta_i\circ \nu^i_V$. Also, define the higher dual exponential map 
   \begin{equation*}
    \exp^*_{(i),L,V}:H^{d-i}(L,V^*(d))\rightarrow H^{d-i}_{\dR}(V^*(d))\oplus H^{d-i-1}_{\dR}(V^*(d)) 
   \end{equation*}
   as the composition of $H^{d-i}(L,V^*(d))\rightarrow H^{d-i}(L,V^*(d)\otimes_{\QQ_p} \BB_{\dR}^\nabla)$ induced by the natural map $V\rightarrow V\otimes_{\QQ_p}\BB_{\dR}^\nabla$ and the inverse of $\nu^{d-i}_{V^*(d)}$. 
   \vs
   
   \noindent {\bf Remark about the filtration.} By construction, it is clear that $\exp_{(i),L,V}$ factors through $H^i\big(L,V\otimes\BB_{\dR}^\nabla\slash \BB_{\dR}^{\nabla +}\big)$. In~\cite{kato99}, Proposition (2.1.10) Kato shows that~\eqref{dualBK} can be refined to
   \begin{equation}\label{filtereddualBK}
    0\rightarrow \BB_{\dR}^{\nabla}\slash\BB_{\dR}^{\nabla +}\rightarrow \BB_{\dR}\slash\BB_{\dR}^0 \rTo^{\nabla} \BB_{\dR}\slash\BB_{\dR}^{-1}\otimes_K\Omegahat^1_K\rTo^{\nabla} \dots\rTo^{\nabla} \BB_{\dR}\slash\BB_{\dR}^{1-d}\otimes_K\Omegahat^{d-1}_K\rightarrow 0.
   \end{equation}
   Define the {\it filtered} de Rham cohomology $H^\bullet_{\dR,\fil}(V)$ of $V$ to be the cohomology of the complex
   \begin{align*}
    0&\rTo \big( V\otimes\BB_{\dR}^\nabla\slash\BB_{\dR}^{\nabla +}\big)^{\calG_K}\rTo \DD_{\dR}(V)\slash\DD_{\dR}^0(V) \rTo^\nabla\DD_{\dR}(V)\slash\DD_{\dR}^{-1}(V)\otimes\Omega^1_K\rTo^\nabla\dots\\
    &\rTo^\nabla \DD_{\dR}(V)\slash \DD_{\dR}^{1-d}(V)\otimes\Omega^{d-1}_K\rTo 0.
   \end{align*}
   A close analysis of the above construction then shows that $\exp_{(i),L,V}$ factors through 
   \begin{equation*}
    \exp_{(i),L,V}: H^i_{\dR,\fil}(V)\oplus H^{i-1}_{\dR,\fil}(V)\rightarrow H^{i+1}(L,V).
   \end{equation*}

%++++++++++++++++++++++++++++++++++++++++++++++++++++++++++++++++++++++++++++++   
%++++++++++++++++++++++++++++++++++++++++++++++++++++++++++++++++++++++++++++++   

%  \subsection{Example: Barsotti-Tate groups}
  
%   In the case when $V$ is the representation associated to a Barsotti-Tate group $G$, the higher exponential maps reduce to the classical Bloch-Kato exponential map.
%   \vs

%++++++++++++++++++++++++++++++++++++++++++++++++++++++++++++++++++++++++++++++   
%++++++++++++++++++++++++++++++++++++++++++++++++++++++++++++++++++++++++++++++   

  \subsection{Duality}
   
   \noindent {\bf Notation.} We keep the notation of Section~\ref{construction}. Let $e^i_V:H^1(L,V\otimes_{\QQ_p}\BB_{\dR}^\nabla)\rightarrow H^i_{\dR}(V)$ be the left inverse of $f^i_V$. As in the proof of Lemma~\ref{maps}, $e^i_V$ is induced from a map $\tilde{e}^i_V: H^1(L,X^{i-1}_V)\rightarrow H^i_{\dR}(V)$. Let $\mu^i_V=e^i_V\oplus g^i_V$, so $\mu^i_V\circ\nu^i_V=\nu^i_V\circ\mu^i_V=\id$. 
   
   \begin{lem}
    For all $1\leq i\leq d$, the wedge product gives a pairing
    \begin{equation*}
     H^i_{\dR}(V)\times H^{d-i-1}_{\dR}(V^*(d))\rightarrow H^{d-1}_{\dR}(\QQ_p(d)).
    \end{equation*}
   \end{lem}
   \begin{proof}
    It is clear that the wedge product gives a pairing 
    \begin{equation*}
     H^0(L,X^i_V)\times H^0(L,X^{d-i-1}_{V^*(d)})\rightarrow H^0(L,X^{d-1}_{\QQ_p(d)}).
    \end{equation*}
    If $x\in H^0(L,X^i_V)$ is of the form $x=\nabla(z)$ for some $z\in H^0(L,X^{i-1}_V)$, then for any $y\in H^0(L,X^{d-i-1}_{V^*(d)})$ we have $x\wedge y=\nabla(z\wedge y)$.
   \end{proof}
   
   \begin{lem}
    The spectral sequence gives an isomorphism $\iota_{L,V}:H^{d-1}_{\dR}(V)\cong H^d(L,V\otimes_{\QQ_p}\BB_{\dR}^\nabla)$.
   \end{lem}
   
   The main result of this section is Proposition~\ref{spectralduality} below. We first need a couple of preliminary lemmas. 
    
   \begin{lem}\label{messy}
    We have short exact sequences
    \begin{align}
     &0\rightarrow X^i_V\rightarrow V\otimes_{\QQ_p}\BB_{\dR}\otimes_K\Omega_K^i\rightarrow X^{i+1}_V\rightarrow 0,\label{eq1}\\
     &0\rightarrow X^{d-i-2}_{V^*(d)}\rightarrow V^*(d)\otimes_{\QQ_p}\BB_{\dR}\otimes_K\Omega_K^{d-i-2}\rightarrow X^{d-i-1}_{V^*(d)}\rightarrow 0,\label{eq2}\\
     &0\rightarrow X^{d-2}_{\QQ_p(d)}\rightarrow \BB_{\dR}(d)\otimes_K\Omega_K^{d-2}\rightarrow \BB_{\dR}(d)\otimes_K\Omega_K^{d-1}\rightarrow 0\label{eq3}
    \end{align}
   \end{lem}
   \begin{proof}
    Clear from Proposition~\ref{spectralseq}.
   \end{proof}
   
   \begin{cor}\label{newmaps}
    We have an isomorphism 
    \begin{align*}
     \iota_{\QQ_p(d)}:&\coker\big(\nabla:\DD_{\dR}(\QQ_p(d))\otimes_K\Omega_K^{d-2}\rightarrow \DD_{\dR}(\QQ_p(d))\otimes_K\Omega_K^{d-1}\big)\\
              &\cong H^2(L,X^{d-2}_{\QQ_p(d)}).
    \end{align*}
   \end{cor}
   \begin{proof}
    The exact sequence~\eqref{eq3} in Lemma~\ref{messy} induces an isomorphism 
    \begin{align*}
     &\coker\big(\nabla:H^1(L,\BB_{\dR}(d)\otimes_K\Omega_K^{d-2})\rightarrow H^1(L,\BB_{\dR}(d)\otimes_K\Omega_K^{d-1})\big)\\
              &\cong H^2(L,X^{d-2}_{\QQ_p(d)})
    \end{align*}
    The above isomorphism is therefore an immediate consequence of Lemma~\ref{katoBdR}.
   \end{proof}
   
   \begin{lem}\label{remainsexact}
    Taking the wedge product of~\eqref{eq1} with $X^{d-i-2}_{V^*(d)}$ gives an exact sequence
    \begin{equation*}
     0\rightarrow X^i_V\wedge X^{d-i-2}_{V^*(d)}\rightarrow V\otimes_{\QQ_p}\BB_{\dR}\otimes_K\Omega_K^i\wedge X^{d-i-2}_{V^*(d)}\rightarrow X^{i+1}_V\wedge X^{d-i-2}_{V^*(d)}\rightarrow 0.
    \end{equation*}
   \end{lem}
   \begin{proof}
    Since $X^i_V\wedge X^{d-i-2}_{V^*(d)}\subset X^{d-2}_{\QQ_p(d)}$ and $V\otimes_{\QQ_p}\BB_{\dR}\otimes_K\Omega_K^i \wedge X^{d-i-2}_{V^*(d)}\subset \QQ_p(d)\otimes\BB_{\dR}\otimes_K\Omega_K^{d-2}$, by~\eqref{eq3} we only have to show surjectivity. Let $x\wedge y\in X^{i+1}_V\wedge X^{d-i-2}_{V^*(d)}$, and choose $\tilde{x}\in V^*(d)\otimes_{\QQ_p}\BB_{\dR}\otimes_K\Omega_K^{d-i-2}$ such that $\nabla(\tilde{x})=x$. Then $\tilde{x}\wedge y\in V\otimes_{\QQ_p}\BB_{\dR}\otimes_K\Omega_K^i\wedge X^{d-i-2}_{V^*(d)}$, and $\nabla(\tilde{x}\wedge y)=x\wedge y$ since $y\in X^{d-i-2}_{V^*(d)}$ and hence $\nabla(y)=0$. 
   \end{proof}
     
   \begin{prop}\label{spectralduality}
    For all $1\leq i\leq d$, the maps $\nu^i_V$ and $\mu^{d-i}_{V^*(d)}$ are dual to each other. More precisely, for all $x\in H^i_{\dR}(V)\oplus H^{i-1}_{\dR}(V)$ and $y\in H^{d-i}(L,V^*(d)\otimes_{\QQ_p}\BB_{\dR}^\nabla)$ we have $\nu^i_V(x)\cup y=\iota_{\QQ_p(d)}(x\wedge \mu^{d-i}_{V^*(d)}(y))$. 
   \end{prop}
   \begin{proof}
    In view of the construction of the maps $\nu^i_V$ and $\mu^{d-i}_{V^*(d)}$ it is sufficient to show that the following two statements are true.
    
    (i) for all $x\in Y^i_{V}$ and $y\in H^1(L,X^{d-i-1}_{V^*(d)})$ the image of $\tilde{f}^i_V(x)\cup y$ in $H^2(L,X^{d-2}_{\QQ_p(d)})$ is $\iota_{\QQ_p(d)}(x_1\wedge \tilde{g}_{d-i,V^*(d)}(y))$;
    
    (ii) for all $z\in Y^{i-1}_V$ and $y\in H^1(L,X^{d-i-1}_{V^*(d)})$ the image of $\tilde{h}^i_V(z)\cup y$ in $H^2(L,X^{d-2}_{\QQ_p(d)})$ is $\iota_{\QQ_p(d)}(z\wedge \tilde{e}_{d-i,V^*(d)}(y))$;
    \vs
    
    Write $\mu^{d-i}_{V^*(d)}(y)=(u_1,u_2)$, so $y=f^{d-i}_{V^*(d)}(u_1)+h^{d-i}_{V^*(d)}(u_2)$. It is then easy to see that $f^i_V(x)\cup f^{d-i}_{V^*(d)}(u_1)=\tilde{h}^i_V(z)\cup h^{d-i}_{V^*(d)}(u_2)=0$.
    \vs 
    
    \noindent {\it Proof of (i).} By the preceding observation, $\tilde{f}^i_V(x)\cup y=\tilde{f}^i_V(x)\cup \tilde{h}^{d-i}_{V^*(d)}(u_2)$. Recall that $\tilde{h}^{d-i}_{V^*(d)}(u_2)$ is given by the cocycle $ u_2\log\chi$. Now $\tilde{f}^i_V$ arises as a connection homomorphism of Galois cohomology, so we deduce from the properties of the cup product that $\tilde{f}^i_V(x)\cup \tilde{h}^{d-i}_{V^*(d)}(u_2)=\delta\big((x\wedge u_2)\log\chi\big)$, where $\delta$ is the connection map $H^1(L,X^{i}_V\wedge X^{d-i-1}_{V^*(d)})\rightarrow H^2(L,X^{i+1}_V\wedge X^{d-i-1}_{V^*(d)})$. We therefore have to show that the image of $\delta\big((x\wedge u_2)\log\chi\big)$ in $H^2(L,X^{d-2}_{\QQ_p(d)})$ is equal to $\iota_{\QQ_p(d)}(x\wedge u_2)$. But this is immediate from the construction of $\iota_{\QQ_p(d)}$. 
    \vs
    
    \noindent {\it Proof of (ii).} Arguing as above, we have $\tilde{h}^i_V(z)\cup y= \tilde{h}^i_V(x)\cup \tilde{f}^{d-i}_{V^*(d)}(u_1)$. The result therefore follows by similar arguments as in (i).
   \end{proof}
   
   Proposition~\ref{spectralduality} has the following important consequence, which establishes a duality between $\exp_{(i),L,V}$ and $\exp_{(i),L,V}^*$.
   
   \begin{cor}
    Let $1\leq i\leq d$. Then for all $x\in Y_i(V)$ and for all $y\in H^{d-i}(L,V^*(d))$ we have
    \begin{equation*}
      \exp_{(i),L,V}(x)\cup y= \delta_d(x\wedge\exp^*_{(i),L,V}(y)).
     \end{equation*}
   \end{cor}

%+++++++++++++++++++++++++++++++++++++++++++++++++++++++++++++++++++++++++++++++++++
%+++++++++++++++++++++++++++++++++++++++++++++++++++++++++++++++++++++++++++++++++++

 \section{The explicit reciprocity law}\label{explicitreciprocity}

  \subsection{Galois cohomology of $\BB^\nabla_{\dR}$}\label{GaloisCohom}
  
   The aim of this section is the proof of Proposition~\ref{vanishingBdR}, which allows us to descend from 
   $V\otimes_{\QQ_p}\BB^\nabla_{\dR}$ to $(V\otimes_{\QQ_p}\BB^\nabla_{\dR})^{\calH_L}$. To prove the propoisition, we closely follow the strategy in Section IV in~\cite{colmez98}, using results from~\cite{brinon03} and~\cite{andreattaiovita07}. The starting point are the following two results.
   
   \begin{prop}\label{citetate}
    For all $i\geq 1$, we have $H^i(L_\infty,\CC_K)=\{1\}$.
   \end{prop}
   \begin{proof}
    See~\cite{hyodo86}.
   \end{proof}
   
   \begin{prop}\label{vanishingGLn}
    For all $i,n\geq 1$, we have $H^i(L_\infty,\GL_n(\CC_K))=\{1\}$.
   \end{prop}
   \begin{proof}
     See the remark after the proof of Lemma 4 in~\cite{brinon03}.
   \end{proof}

   \begin{lem}\label{trivialcocyc}
    For all $k\in\{1,2,\dots,+\infty\}$, we have
    \begin{equation*}
     H^1(L_\infty,\GL_n(\BB_{\dR}^{\nabla +}\slash\fil^k\BB_{\dR}^{\nabla +}))=\{1\}.
    \end{equation*}
   \end{lem}
   \begin{proof}
    The case $k=1$ corresponding to $\BB_{\dR}^{\nabla +}\slash\fil^1\BB_{\dR}^{\nabla +}\cong\CC_K$  is
    shown in Proposition~\ref{citetate}. The case $k=+\infty$ can be obtained from the case $k\in\NN$ by
    passing to the direct limit. On the other hand, if $k\geq 1$, then we have a short exact sequence 
    \begin{equation*}
     1\rightarrow 1+ M_n\big(\fil^k\BB_{\dR}^{\nabla +}\slash\fil^{k+1}\BB_{\dR}^{\nabla +}\big) \rightarrow
     \GL_n\big(\BB_{\dR}^{\nabla +}\slash\fil^{k+1}\big)\rightarrow \GL_n\big(\BB_{\dR}^{\nabla +}\slash\fil^k\big)
     \rightarrow 1.
    \end{equation*}
    Taking $\calH_L$-cohomology reduces the result to showing that 
    \begin{equation}\label{vanishMd}
     H^1\big(L_\infty,1+ M_n(\fil^k\BB_{\dR}^{\nabla +}\slash\fil^{k+1}\BB_{\dR}^{\nabla +})\big)=\{1\}.
    \end{equation} 
    Since we have an isomorphism of $\calH_L$-modules $1+ M_n(\fil^k\BB_{\dR}^{\nabla +}\slash\fil^{k+1}\BB_{\dR}^{\nabla +})\cong
    M_n(\CC_K)$,~\eqref{vanishMd} follows from Proposition~\ref{citetate} above.
   \end{proof}

   As a corollary, we obtain the following result, which is an analogue of Th\'eor\`eme IV.2.1 in~\cite{colmez98}. 
   
   \begin{prop}\label{BdRmod}
    If $V$ is a $p$-adic representation of $\calG_L$, then $(\BB_{\dR}^{\nabla +}\otimes V)^{\calH_L}$ is a free $(\BB_{\dR}^{\nabla +})^{\calH_L}$-module of rank $\dim_{\QQ_p}V$.
   \end{prop}
   \begin{proof}
    By choosing a basis $e_1,\dots,e_n$ of $V$ over $\QQ_p$ we can consider the representation as a continuous $1$-cocycle $\tau\rightarrow U_\tau:H_K\rightarrow \GL_n(\QQ_p)\subset\GL_n\big(\BB_{\dR}^{\nabla +}\big)$. But this cocycle is trivial by Proposition~\ref{vanishingGLn}, which implies that the $\calH_L$-module $\BB_{\dR}^{\nabla +}\otimes V$ is isomorphic to $(\BB_{\dR}^{\nabla +})^n$.
   \end{proof}
   
   \begin{cor}\label{BdRvs}
    If $V$ is a $p$-adic representation of $\calG_L$, then $(\BB_{\dR}^{\nabla}\otimes V)^{\calH_L}$ is a $(\BB_{\dR}^{\nabla})^{\calH_L}$-vector space of dimension $\dim_{\QQ_p}V$.
   \end{cor}
   
   \begin{prop}\label{vanishingBdR}
    Let $i\geq 1$. Then
    
    (i)  $H^i\big(L_\infty,V\otimes_{\QQ_p}\BB^{\nabla +}_{\dR}\slash\fil^k\BB^{\nabla +}_{\dR}\big)=\{1\}$ 
    for all $k\in\{ 1,2,\dots,+\infty\}$. 
    
    (ii) $H^i\big(L_\infty,V\otimes_{\QQ_p}\BB^\nabla_{\dR}\slash \BB^{\nabla +}_{\dR}\big)=\{1\}$.
    
    (iii) $H^i\big(L_\infty,V\otimes_{\QQ_p}\BB^\nabla_{\dR}\big)=\{1\}$.
   \end{prop}
   \begin{proof}
    By Proposition~\ref{BdRmod}, the $\calH_L$-module $\BB_{\dR}^{\nabla +}\otimes V$ is isomorphic to $(\BB_{\dR}^{\nabla +})^n$. It is therefore sufficient to prove the results when $V=\QQ_p$. If $k$ is finite, then (i) follows from the observation that $\BB^{\nabla +}_{\dR}\slash\fil^k\BB^{\nabla +}_{\dR}\cong \CC_K^k$ as $\calH_L$-modules. The case when $k=+\infty$ follows from the finite case by taking inverse limit, observing that the natural topology on $\BB_{\dR}^\nabla$ is weaker than the $t$-adic topology. (ii) can be  deduced from (i) using the isomorphism between $\fil^{-k}(\BB_{\dR}^\nabla\slash\BB^{\nabla +}_{\dR})$ and $t^{-k}(\BB_{\dR}^{\nabla +}\slash \fil^k\BB^{\nabla +}_{\dR})$ and passing to the direct limit over $k$. Similarly, (iii) can be deduced from (i) by writing $\BB_{\dR}^\nabla=\varinjlim t^{-k}\BB_{\dR}^{\nabla +}$.
   \end{proof}

   \begin{cor}\label{isomH-inf}
    For all $i\geq 1$, the inflation map gives an isomorphism
    \begin{equation*}
     H^i(L,\BB^\nabla_{\dR}\otimes_{\QQ_p}V)\cong H^i(G_L,(V\otimes_{\QQ_p}\BB^\nabla_{\dR})^{\calH_L}).
    \end{equation*}
   \end{cor}
   \begin{proof}
    Immediate consequence of (iii) in Proposition~\ref{vanishingBdR} and the Hochschild-Serre spectral sequence. 
   \end{proof}

%+++++++++++++++++++++++++++++++++++++++++++++++++++++
%+++++++++++++++++++++++++++++++++++++++++++++++++++++
 
  \subsection{Galois cohomology of $\DD(V)^{\psi=0}$}\label{Koszulcohom}

   To simplify the notation, throughout this section we drop the indices of $\Gamma_L$ and $H_L$. Let $V$ be a $p$-adic representation of $\calG_K$. The main result of this section is the following proposition, which can be seen as a generalisation of Proposition I.5.1 in~\cite{cherbonniercolmez99}.
   
   \begin{prop}\label{zerocohom}
    $D(V)^{\psi=0}$ has trivial $G_L$-cohomology.
   \end{prop}
   
   A crucial ingredient in the proof is the following result from~\cite{andreattabrinon06}.
   
   \begin{prop}\label{quoteAB}
    The module $\DD(V)^{\psi=0}$ admits a decomposition of $G_L$-modules
    \begin{equation*}
     \DD(V)^{\psi=0}=\DD(V)_0\oplus\dots\oplus\DD(V)_{d-1}
    \end{equation*}
    such that (1) $\gamma-1$ is invertible on $\DD(V)_0$, and
    
    \noindent (2) $h_i-1$ is invertible on $\DD(V)_i$ for $1\leq i\leq d-1$.
   \end{prop}
   \begin{proof}
    Proposition 4.44 in~\cite{andreattabrinon06}.
   \end{proof}
   
   \begin{remark}\label{invertibleoverconvergence}
    In fact, the proof of Proposition 4.46 in~\cite{andreattabrinon06} shows that there exists $r_V'> 0$ such that for all $r\leq r_V'$, we have $(\gamma-1)^{-1}$ (resp. $(h_i-1)^{-1}$ for $1\leq i<d$) gives a continuous map $\DD^{(0,r]}(V)_0\rightarrow \DD^{(0,r]}(V)_0$ (resp. $\DD^{(0,r]}(V)_i\rightarrow \DD^{(0,r]}(V)_i$).
   \end{remark}
   
   \begin{lem}\label{topologicalgroup}
    For all $0\leq i\leq d-1$, $\DD(V)_i$ is isomorphic as a topological group to the projective limit of additive discrete $p$-groups. 
   \end{lem}
   \begin{proof}
    It is sufficient to show that $\DD(V)_i$ is complete in the weak topology. Since $\psi$ is continuous for the weak topology, $\DD(V)^{\psi=0}$ is a closed subgroup of $\DD(V)$ and hence by Proposition~\ref{projlim} complete in the weak topology. As shown in the proof of Proposition 4.44 in~\cite{andreattabrinon06}, $\DD(V)_i$ is defined as the kernel of a continuous map $\DD(V)^{\psi=0}\rightarrow \DD(V)^{\psi=0}$ and hence is closed and complete in the weak topology.
   \end{proof}
   
   In order to prove Proposition~\ref{zerocohom}, it is therefore sufficient to show that $\DD(V)_i$ has trivial $G_L$-cohomology for all $0\leq i\leq d-1$. We first prove the result for $i=0$. By the Hochschild-Serre spectral sequence it is sufficient to show that for all $0\leq i\leq d-1$, $H^i(H,\DD(V)_0)$ has trivial $\Gamma$-cohomology (c.f. Corollary~\ref{trivial} below), which is equivalent to the statement that the map $\gamma-1:H^i(H,\DD(V)_0)\rightarrow H^i(H,\DD(V)_0)$ is an isomorphism for all $i$. The proof of this result relies on the observation that the $H$-cohomology of $\DD(V)_0$ can be calculated using the {\it Koszul complex}.
   \vs
   
   \noindent {\bf Definition.} Let $M$ be a module equipped with a continuous action of $H$. For $1\leq i\leq d-1$, define the complex
   \begin{equation*}
    \calT^\bullet_i(M): 0\rTo M\rTo^{h_i-1} M\rTo 0.
   \end{equation*}
   Then the Koszul complex $\calK^\bullet(M)$ of $M$ with respect to $h_1-1,\dots,h_{d-1}-1$ is the tensor product in the category of $R$-complexes of the complexes $\calT^\bullet_i(M)$. Note that $\calK^\bullet(M)\cong M\otimes_{\Lambda(H)}\calK^\bullet(\Lambda(H))$. (For a detailed exposition, see Section 4.5 in~\cite{weibel94}).
   
   \begin{prop}\label{koszul}
    If $M$ is a continuous $H$-module which is isomorphic (as topological group) to the projective limit of additive discrete $p$-groups, then $H^i(\calK^\bullet(M))\cong H^i(H,M)$ for all $1\leq i\leq d-1$. 
   \end{prop}
   \begin{proof}
    Note that $(h_1-1,\dots,h_{d-1}-1)$ is a regular sequence in $\Lambda(H)$, so $\calK^\bullet(\Lambda(H))$  is a free resolution of $\Lambda(H)_H\cong\ZZ_p$. It follows that the complex $\calC(H,M)=\Hom_{\Lambda(H)}\big(\calK^\bullet(\Lambda(H)), M\big)$ calculates the continuous $H$-cohomology of $M$ (c.f. equation (18) in~\cite{lazard65}). But the symmetry of the Koszul complex implies that $H^i(\calC(H,M))\cong H_{d-i-1}(\calK^\bullet(M))$ for all $0\leq i\leq d-1$, which finishes the proof.
  \end{proof}
  
   Explicitly, the Koszul complex can be constructed inductively in the following way (c.f.~\cite{eisenbud95}): let $\calK^\bullet_1(M): M\rTo^{h_1-1}M$. For $i\geq 2$, define $\calK^\bullet_i(M)$ to be the mapping cone of $\calK^\bullet_{i-1}(M)\rTo^{h_i-1} \calK^\bullet_{i-1}(M)$. Then $\calK^\bullet(M)=\calK^\bullet_{d-1}(M)$. Note that by construction the complex $\calK^\bullet(M)$ looks like
   \begin{equation*}
    0\rTo M\rTo^{F_1} M^{\oplus \tbinom{d-1}{1}} \rTo^{F_2} M^{\oplus \tbinom{d-1}{2}}\rTo^{F_3}\dots\rTo^{F_d}M \rTo 0,
   \end{equation*}
   where the maps $F_j$ are defined as follows: For $0\leq i<d$, let $n_i=\tbinom{d-1}{i}$, and let $\frak{A}_i$ be the collection of all subsets of the set $\{1,\dots,d-1\}$ of cardinality $i$. Note that since $\# \frak{A}_i=n_i$, we can index $M^{\oplus n_i}$ by $\frak{A_i}$, i.e. $M^{\oplus n_i}=\bigoplus_{A\in\frak{A}_i} M$, and if $\frak{m}\in M^{\oplus n_i}$, then $\frak{m}=(m_A)_{A\in\frak{A}_i}$. % Define a complete order on $\frak{A}_i$ by saying that if $A,B\in\frak{A}_i$, $A=\{a_1,\dots,a_i\}$ and $B=\{b_1,\dots,b_i\}$ with $a_1<\dots<a_i$ and $b_1<\dots<b_i$, then $A<B$ if either $a_1<b_1$ or there exists $1\leq i'<i$ such that $a_k=b_k$ for all $1\leq k\leq i'$ and $a_{i'+1}=b_{i'+1}$.
   Fix $1\leq j\leq d$. For $C\in\frak{A}_j$, say $C=\{k_1,\dots,k_j\}$, define $f_C: M^{\oplus n_{j-1}}\rightarrow M,$ 
   \begin{equation*}
    (m_A)_{A\in\frak{A}_{j-1}}  \rTo \sum_{l=1}^j(-1)^{k_l}(h_{k_l}-1)m_{C\backslash\{k_l\}}
   \end{equation*}
   Define $F_j=\bigoplus_{C\in\frak{A}_j}f_C$. From this description of the Koszul complex, it is easy to see that the following lemma is true.
   
   \begin{lem}\label{quasi-isom}
    Let $h_1',\dots,h_{d-1}'\in\Lambda(H)$ such that $h_i$ differs from $h_i'$ by an element in $\Lambda(H)^\times$ for all $1\leq i<d$. Then the Koszul complexes $\calK^\bullet(M)$ and $\calK'^\bullet(M)$ obtained from $h_1,\dots,h_{d-1}$ and $h_1'\dots,h_{d-1}'$ are quasi-isomorphic. Moreover, for all $0\leq j\leq d$ there exists $G^{(j)}=\big(G^{(j)}_{A\in\frak{A}_j}\big)\in\big(\Lambda(H)^\times\big)^{\oplus n_j}$ such that the quasi-isomorphism is given by
    \begin{diagram}
     0& \rTo & M & \rTo^{F'_1} & M^{\oplus \tbinom{d-1}{1}} & \rTo^{F'_2} & M^{\oplus \tbinom{d-1}{2}}& \rTo^{F'_3}& \dots& \rTo^{F'_d}& M & \rTo & 0 \\
      &      & \dTo^{G^{(0)}} & & \dTo^{G^{(1)}}             &            & \dTo^{G^{(2)}}             &           &      &    & \dTo^{G^{(d)}} & & \\
     0& \rTo & M & \rTo^{F_1} & M^{\oplus \tbinom{d-1}{1}} & \rTo^{F_2} & M^{\oplus \tbinom{d-1}{2}}& \rTo^{F_3}& \dots& \rTo^{F_d}& M & \rTo & 0 
    \end{diagram}
   \end{lem}
   
   \begin{lem}
    Let $M$ be a module with a continuous action of $G_K$. If $\gamma-1$ is invertible on $M$, then so is $\gamma\omega-1$ for any $\omega\in\Lambda(H)^{\times}$.
   \end{lem}
   \begin{proof}
    Assume without loss of generality that $\omega\cong 1\mod \frak{M}$, where $\frak{M}$ denotes the unique maximal ideal in $\Lambda(H)$. Note that 
    \begin{equation*}
     \gamma\omega-1= (\gamma-1)\omega\Big[1+(\gamma-1)^{-1}\frac{\omega-1}{\omega}\Big],
    \end{equation*}
    so it is sufficient to show that $1+(\gamma-1)^{-1}(1-\omega^{-1})$ is invertible on $M$, i.e. that the formal expansion of $\big( 1+(\gamma-1)^{-1}(1-\omega^{-1})\big)^{-1}$ converges on $M$. But this is immediate since $\omega-1\in\frak{M}$ and the action of $G_K$ on $M$ is continuous. 
   \end{proof}
   
   \begin{prop}\label{Gammacohom}
    For all $0\leq j\leq d-1$, the cohomology group $H^j(H,\DD(V)_0)$ has trivial $\Gamma$-cohomology.
   \end{prop}
   \begin{proof}
    To simplify the notation, let $M=\DD(V)_0$. The action of $\gamma$ defines a homomorphism of complexes
    \begin{diagram}
     0& \rTo & M & \rTo^{F_1} & M^{\oplus \tbinom{d-1}{1}} & \rTo^{F_2} & M^{\oplus \tbinom{d-1}{2}}& \rTo^{F_3}& \dots& \rTo^{F_d}& M & \rTo & 0 \\
      &      & \dTo^{\gamma} & & \dTo^{\gamma}             &            & \dTo^{\gamma}             &           &      &    & \dTo^{\gamma} & & \\
     0& \rTo & M & \rTo^{F'_1} & M^{\oplus \tbinom{d-1}{1}} & \rTo^{F'_2} & M^{\oplus \tbinom{d-1}{2}}& \rTo^{F'_3}& \dots& \rTo^{F'_d}& M & \rTo & 0 
    \end{diagram}
    where the lower complex is the Koszul complex with respect to elements $h_1',\dots,h_{d-1}'\in\Lambda(H)$ such that $h_i'$ differs from $h_i$ by a unit in $\Lambda(H)$ for all $1\leq i<d$. But by Lemma~\ref{quasi-isom} the lower complex is quasi-isomorphic to the original complex $\calK^\bullet(M)$ via maps $G^{(j)}=\big(G^{(j)}_A\big)_{A\in\frak{A}_j}\in\big(\Lambda(H)^\times\big)^{\oplus n_j}$. The action of $\gamma$ on $\calK^\bullet(M)$ is therefore given by 
    \begin{diagram}
     0& \rTo & M & \rTo^{F_1} & M^{\oplus \tbinom{d-1}{1}} & \rTo^{F_2} & M^{\oplus \tbinom{d-1}{2}}& \rTo^{F_3}& \dots& \rTo^{F_d}& M & \rTo & 0 \\
      &      & \dTo^{\gamma G^{(0)}} & & \dTo^{\gamma G^{(1)}}             &            & \dTo^{\gamma G^{(2)}}             &           &      &    & \dTo^{\gamma G^{(d)}} & & \\
     0& \rTo & M & \rTo^{F_1} & M^{\oplus \tbinom{d-1}{1}} & \rTo^{F_2} & M^{\oplus \tbinom{d-1}{2}}& \rTo^{F_3}& \dots& \rTo^{F_d}& M & \rTo & 0 
    \end{diagram}
    so $\gamma-1$ acts on $H^j(H,M)$ by sending an element $\frak{m}=(m_A)_{A\in\frak{A}_j}\in M^{\oplus n_j}$ which satisfies $F_j(\frak{m})=0$ to $(\gamma G^{(j)}-1)\frak{m}=\big( (\gamma G^{(j)}_A-1)m_A\big)_{A\in\frak{A}_j}$. Define a map $\Delta:M^{\oplus n_j}\rightarrow M^{\oplus n_j}$ by sending $\frak{n}=(n_A)_{A\in\frak{A}_j}$ to $\big( (\gamma G^{(j)}_A-1)^{-1}m_A\big)_{A\in\frak{A}_j}$. 
    
    \noindent {\it Claim.} $\Delta$ factors through $\image(F_{j-1})$ and is an inverse of $\gamma-1$ on $H^j(H,M)$.
    
    \noindent {\it Proof of claim.} It is clear from the definition that if $\Delta$ factors through $\image(F_{j-1})$ and restricts to a map $H^j(H,M)\rightarrow H^j(H,M)$, then it is an inverse of $\gamma-1$. Suppose that $\frak{n}=F_{j-1}(\frak{l})$ for some $\frak{l}=(l_C)_{C\in \frak{A}_{j-1}}$, so for all $A\in\frak{A}_j$, say $A=\{k_1,\dots, k_j\}$, we have $n_A=\sum_{i=1}^j(-1)^{k_i}(h_{k_i}-1)l_{A\backslash\{k_i\}}$. Explicit calculation shows that
    \begin{equation*}
     (\gamma G^{(j)}_A-1)^{-1}(h_{k_i}-1)=(h_{k_i}-1)(\gamma G^{(j)}_{A\backslash \{k_i\}}-1)^{-1},
    \end{equation*}
    which implies that $\Delta$ factors through $\image(F_{j-1})$. 
    
    Suppose now that $\frak{m}=(m_A)_{A\in\frak{A}_j}$ satisfies $F_j(\frak{m})=0$, so for all $C\in\frak{A}_{j+1}$, say $C=\{k_1,\dots,k_{j+1}\}$, we have $\sum_{i=1}^{j+1}(-1)^{k_i}(h_{k_i}-1)m_{C\backslash\{k_i\}}=0$. Again, explicit calculation shows that 
    \begin{equation*}
     (h_{k_i}-1)(\gamma G^{(j)}_{C\backslash\{k_i\}}-1)^{-1}=(\gamma G^{(j)}_C-1)^{-1}(h_{k_i}-1),
    \end{equation*}
    so $\Delta$ restricts to $H^j(H,M)\rightarrow H^j(H,M)$.
   \end{proof}
        
   \begin{cor}\label{trivial}
    $\DD(V)_0$ has trivial $G_L$-cohomology.
   \end{cor}
   \begin{proof}
    Since $\cd_p(\Gamma)=1$, the Hochschild-Serre spectral sequence degenerates to give a short exact sequence
    \begin{equation*}
     0\rightarrow H^1(\Gamma,H^{i-1}(H,\DD(V)_0)) \rTo H^i(G,\DD(V)_0) \rTo H^i(H,\DD(V)_0)^{\Gamma} \rightarrow 0
    \end{equation*}
    for all $1\leq i\leq d$. The result is therefore an immediate consequence of Lemma~\ref{Gammacohom}.
   \end{proof}
   
   Assume now that $i> 0$. Let $H_i=\overline{\langle h_i\rangle}$. 
   
   \begin{prop}\label{nodividesp}
    $\DD^\dagger(V)_i$ has trivial $H$-cohomology for all $1\leq i<d$.
   \end{prop}
   \begin{proof} 
    To simplify the notation, assume without loss of generality that $i=1$, and let $M=\DD(V)_1$.  A special case of Proposition~\ref{koszul} implies that the $H_1$-cohomology of $M$ is given by the complex $M\rTo^{h_1-1}M$, so by Proposition~\ref{quoteAB} $M$ has trivial $H_1$-cohomology. Note that $H=\bigoplus_{j=1}^{d-1} H_j$. For $1\leq k\leq d-1$ let $H_1^k=\bigoplus_{j=1}^k H_j$, so $H^k_1\slash H^{k-1}_1\cong H_k$. Then for all $1\leq j\leq k$ the Hochschild-Serre degenerates to the short exact sequence
    \begin{equation*}
     0\rTo H^1(H_k,H^{j-1}(H_1^{k-1},M))\rTo H^j(H_1^k,M)\rTo H^j(H^{k-1}_1,M)^{H_k}\rTo 0.
    \end{equation*}
    Induction on $k$ therefore implies that $M$ has trivial $H^k_1$-cohomology, which finishes the proof.
   \end{proof}
    
   \begin{cor}
    $\DD(V)_i$ has trivial $G$-cohomology.
   \end{cor}
   \begin{proof}
    Since $\cd_p(\Gamma)=1$, the Hochschild-serre spectral sequence degenerates to give a short exact sequence
    \begin{equation*}
     0\rTo H^1\big(\Gamma,H^{j-1}(H,\DD(V)_i)\big) \rTo H^j(G_K,\DD(V)_i) \rTo H^j(H,\DD(V)_i)^\Gamma \rTo 0
    \end{equation*}
    for all $1\leq j\leq d$. The result is therefore an immediate consequence of Proposition~\ref{nodividesp}.
   \end{proof}
   
   \begin{remark}\label{valuesoverconvergent}
    Let $1\leq i\leq d$, and $c: G_L^{\times i}\rightarrow \DD(V)^{\psi=0}$ be a continuous cocycle with values in $\DD^{(0,r]}(V)^{\psi=0}$ for $r\leq r_V'$. The proof of Proposition~\ref{zerocohom} and Remark~\ref{invertibleoverconvergence} then imply that $c$ is equal to a coboundary $db$, where $b:G_L^{\times i-1}\rightarrow \DD^{(0,r]}(V)^{\psi=0}$.
   \end{remark}

%+++++++++++++++++++++++++++++++++++++++++++++++++++++++++++++++++++++++++++++++   
%+++++++++++++++++++++++++++++++++++++++++++++++++++++++++++++++++++++++++++++++   
   
  \subsection{The explicit reciprocity law}\label{EL}
  
   Let $V$ be a $p$-adic representation of $\calG_L$. For $f=\phi$ of $\psi$, write $\calT_f^\bullet(\DD(V))$ for the absolute complex of the mapping cone associated to $f-1:\calC^\bullet(G_L,\DD(V))\rTo \calC^\bullet(G_L,\DD(V))$. Denote by $\calZ^i(G_L,\DD(V))$ (resp. by $\calZ^i_f(G_L,\DD(V))$) the group of continuous $i$-cocycles of the complex $\calC^\bullet(G_L,\DD(V))$ (resp. $\calT_f^\bullet(\DD(V))$). Define $\calB^i(G_L,\DD(V))$ and $\calB^i_f(G_L,\DD(V))$ to be the respective groups of $i$-coboundaries. Also, denote the differentials of the complex $\calC^\bullet(G_L,\DD(V))$ by $\partial^i$. 
   
   \begin{prop}\label{cohomisom}
    For all $0\leq i\leq d+1$ we have a canonical isomorphism
    \begin{equation*}
     H^i\big(\calT_\psi^\bullet(\DD(V))\big) \cong H^i\big(\calT_\phi^\bullet(\DD(V))\big) \cong H^i(L,V).
    \end{equation*}
   \end{prop}
   \begin{proof}    
    The isomorphism $H^i\big(\calT_\phi^\bullet(\DD(V))\big) \cong H^i(L,V)$ is the content of Theorem 3.3 in~\cite{andreattaiovita07} (see also Theorem 1.1 in~\cite{morita08}. It therefore remains to construct an isomorphism $\alpha_i:H^i\big(\calT_\psi^\bullet(\DD(V))\big) \cong H^i\big(\calT_\phi^\bullet(\DD(V))\big)$. By construction, an element in $H^i\big(\calT_\psi^\bullet(\DD(V))\big)$ is given by a pair $(x,y)$ with $x\in \calB^{i-1}(G_L,\DD(V))$ and $y\in \calZ^i(G_L,\DD(V))$ such that $(\psi-1)y=\partial^{i-1}(x)$. Then $(\phi\psi-1)y\in \calZ^i(G_L,\DD(V)^{\psi=0})$. By Proposition~\ref{zerocohom}, there exists a continuous map $z:\calC^{i-1}(G_L,\DD(V)^{\psi=0})$ such that $\partial^{i-1}(z)=(\phi\psi-1)y$. Define $\alpha_i(x,y)=(-\phi(x)+z,y)$. It is easy to check that $(\phi-1)y=\partial^{i-1}\big(-\phi(x)+z\big)$. If $y=\partial^{i-1}(y')$ and $x=(\psi-1)y'+\partial^{i-2}(x')$ for some $x'\in\calC^{i-2}(G_L,\DD(V))$ and $y'\in \calC^{i-1}(G_L,\DD(V))$, then 
    \begin{align*}
     \alpha_i(x,y) & =\big(-\phi\circ(\psi-1)y'-\phi\circ\partial^{i-2}(x')+(\phi\psi-1)y',\partial^{i-1}(y')\big) \\
     & = \big((\phi-1)y'+\partial^{i-2}(\phi(x')), \partial^{i-1}(y')\big).
    \end{align*}
    It therefore follows that $\alpha_i$ gives a map $H^i\big(\calT_\psi^\bullet(\DD(V))\big) \rightarrow H^i\big(\calT_\phi^\bullet(\DD(V))\big)$. It remains to show that $\alpha_i$ is an isomorphism. Let $u\in \calB^{i-1}(G_L,\DD(V))$ and $v\in \calZ^i(G_L,\DD(V))$ satisfy $(\phi-1)v=\partial^i(u)$. Define $\beta_i(u,v)=(-\psi(u),v)$. It is easy to see that $\beta_i$ defines a map $H^i\big(\calT_\phi^\bullet(\DD(V))\big) \rightarrow H^i\big(\calT_\psi^\bullet(\DD(V))\big)$. If $(x,y)$ defines an element in $H^i\big(\calT_\psi^\bullet(\DD(V))\big)$, then $\beta_i\circ\alpha_i(x,y)-(x,y)=(-\psi(z),0)=(0,0)$ since $z$ has values in $\DD(V)^{\psi=0}$. Similarly, if $(u,v)$ defines a cohomology class in $H^i\big(\calT_\phi^\bullet(\DD(V))\big)$, then 
    \begin{equation*}
     \alpha_i\circ\beta_i(u,v)-(u,v) = \big((\phi\psi-1)(u)+z,0\big).
    \end{equation*}
    Now $w=(\phi\psi-1)(u)+z$ has values in $\DD(V)^{\psi=0}$ and satisfies $\partial^{i-1}(w)=0$, so by Proposition~\ref{zerocohom} there exists $w'\in\calC^{i-2}(G_L,\DD(V)^{\psi=0})$ such that $w=\partial^{i-2}(w')$. But then by definition $(w,0)$ is a coboundary and hence zero in $H^i\big(\calT_\phi^\bullet(\DD(V))\big)$. We deduce that $\alpha_i\circ\beta_i=\beta_i\circ\alpha_i=\id$, which finishes the proof.
   \end{proof}
   
   The isomorphism $H^i\big(\calT_\phi^\bullet(\DD(V))\big) \cong H^i(L,V)$ is explicitly given as follows: let $(\alpha,\beta)$ be an $i$-cochain of $\calT_\phi^\bullet(\DD(V))$, which is given by an element in $\calC^{i-1}(G_L,\DD(V))\times \calC^i(G_L,\DD(V))$. Define 
   \begin{equation*}
    c_{\alpha,\beta}=\beta+(-1)^id\big(\sigma(\alpha)\big),
   \end{equation*}
   where $\sigma$ is a left inverse of $\phi-1$ (c.f. Appendix II in~\cite{andreattaiovita07}). Recall that the derivation on $\calT_\phi^\bullet(\DD(V))$ is given by $d\big((\alpha,\beta)\big)=\big((-1)^i(\phi-1)\beta+d\alpha,d\beta\big)$. We deduce that $dc_{\alpha,\beta}=0$, so $c_{\alpha,\beta}$ is a cocycle in $\calC^i(\calG_L,\DD(V))$. It is easy to check that $(\phi-1)c_{\alpha,\beta}=0$, so in fact $c_{\alpha,\beta}$ has values in $V$. As shown in Section 3.4 in~\cite{andreattaiovita07}, the isomorphism in Proposition~\ref{cohomisom} is given by the map $(\alpha,\beta)\rightarrow c_{\alpha,\beta}$.
   \vs
   
   An important corollary of Proposition~\ref{cohomisom} is the following Proposition:
  
   \begin{prop}\label{hom}
    For all $1\leq i\leq d$, we have a canonical homomorphism
    \begin{equation*}
     \delta^{(i)}: H^i(G_L,D(V)^{\psi=1}) \rTo H^i(L,V).
    \end{equation*}
   \end{prop}
   \begin{proof}
    If $\tilde{\beta}$ is a cocycle representing a cohomology class in $H^i(G_L,\DD(V)^{\psi=1})$, then the pair $(0,\tilde{\beta})$ defines an element in the $i$th cohomology group of the complex $\calT^\bullet_\psi(\DD(V))$.
   \end{proof}

   \noindent {\bf Remark.} Explicitly, a cocycle representing $\delta^{(i)}(\tilde{\beta})$ is constructed as follows: the map $\beta=(\phi-1)\tilde{\beta}$ is a cocycle with values in $\DD(V)^{\psi=0}$ and defines a cohomology class in $ H^i(G_L,\DD(V)^{\psi=0})$. But  $\DD(V)^{\psi=0}$ has trivial $G_L$-cohomology by Proposition~\ref{zerocohom}, so $\beta$ is equal to a coboundary $d(\alpha)$, where $\alpha$ is a continuous map $G_L^{\times (i-1)}\rightarrow \DD(V)^{\psi=0}$. Then $(\alpha,\beta)$ is an $i$-cochain of $\calT_\phi^\bullet(\DD(V))$
   %i.e. an element in $\calC^{i-1}(G_L,\DD(V))\times \calC^{i}(G_L,\DD(V))$, 
   and hence defines an element in $H^i(L,V)$ by Proposition~\ref{cohomisom}. Explicitly, if $\sigma$ is a left inverse of $\phi-1$, then $\delta^{(i)}(\frak{c})$ is given by the cocycle
    \begin{equation*}
     c_{\alpha,\beta}=\beta+(-1)^i d\big(\sigma(\alpha)\big).
    \end{equation*}
   
   \begin{lem}
    Let $r\leq r_V$, and choose $\alpha$ such that $\image(\alpha)\subset \DD^{(0,r]}(V)$ (c.f. Remark~\ref{valuesoverconvergent}). Then $\image(\sigma(\alpha))\subset V\otimes_{\QQ_p}\tilde{\AA}^{(0,pr]}$.
   \end{lem}
   \begin{proof}
    See the proof of Proposition 8.1 in~\cite{andreattaiovita07}.
   \end{proof}

   \begin{thm}\label{samecohomology}
    Let $V$ be a de Rham representation of $\calG_L$, and let $n\in\NN$ such that $p^{-n}<\max\{r_V,r_V'\}$. Then for all $1\leq i\leq d$, we have a commutative diagram
    \begin{diagram}
     H^i(G_L,D(V)^{\psi=1}) & \rTo^{\phi^{-n}} & H^i(G_L,(V\otimes \BB_{\dR}^\nabla)^{\calH_L}) \\
     \dTo^{\delta^{(i)}} & & \dTo^{\inf} \\
     H^i(L,V) & \rTo^{V\rightarrow V\otimes\BB_{\dR}^\nabla} &  H^i(L,V\otimes \BB_{\dR}^\nabla) 
    \end{diagram}
   \end{thm}
   \begin{proof}
    Let $\frak{c}$ be a cocycle representing an element in $H^i(G_L,\DD(V)^{\psi=1})$. In the notation of the remark following Proposition~\ref{hom}, the image of $\frak{c}$ under $\delta^{(i)}$ is given by the cocycle $c_{\alpha,\beta}=\beta+(-1)^i d\big(\sigma(\alpha)\big)$. Let $r=\max\{r_V,r_V'\}$. Then both $\beta$ and $\sigma(\alpha)$ have values in $V\otimes\AA^{(0,r]}$, so both $\phi^{-n}(\beta)$ and $\phi^{-n}(\sigma(\alpha))$ can be seen as cocycles with values in $V\otimes\BB_{\dR}^\nabla$, and $c'_{\alpha,\beta}=\phi^{-n}(\alpha)$ as a cocycle with values in $V\otimes\BB_{\dR}^\nabla$ which differs from the cocycle
    \begin{equation*}
     (g_1,\dots,g_i)\rightarrow \phi^{-n}\big(\beta(g_1,\dots,g_i)\big)
    \end{equation*}
    by the coboundary $(g_1,\dots,g_i)\rightarrow \phi^{-n}\big(d\sigma\alpha(g_1,\dots,g_i)\big)$. 
   \end{proof}
   
   By Corollary~\ref{isomH-inf} the inflation map gives an isomorphism $H^i(G_L,(V\otimes \BB_{\dR}^\nabla)^{\calH_L})\cong H^i(L,V\otimes \BB_{\dR}^\nabla)$. In terms of the relative dual exponential maps, Theorem~\ref{samecohomology} can be restated as follows.
   
   \begin{thm}
    Let $V$ be a de Rham representation of $\calG_K$, and let $n\in\NN$ such that $p^{-n}>\max\{r_V,r_V'\}$. Then for all $1\leq i\leq d$, we have a commutative diagram
    \begin{diagram}
     H^i(G_L,D(V)^{\psi=1}) & \rTo^{\phi^{-n}} & H^i(L,V\otimes \BB_{\dR}^\nabla) \\
     \dTo^{\delta^{(i)}} & & \dTo^{\cong} \\
     H^i(L,V) & \rTo^{\exp^*_{(d-i),K,V^*(d)}} &  H^i_{\dR}(V)\oplus H^{i-1}_{\dR}(V)
    \end{diagram}
   \end{thm}
   
%+++++++++++++++++++++++++++++++++++++++++++++++++++++++++++++++++++++=======+++++++++++++++++++++++++++++++++++++++++++++++++++
%+++++++++++++++++++++++++++++++++++++++++++++++++++++++++++++++++++++++++++++++++++++++++++++++++++++++++++++++++++++++++++++++

\bibliography{references}

\vspace{7ex}

 Sarah Livia Zerbes
 
 Department of Mathematics
 
 Imperial College London
 
 London SW7 2AZ
 
 United Kingdom
 
 {\it email:} s.zerbes@imperial.ac.uk

\end{document}